\documentclass[12pt]{amsart}
\usepackage{amscd,amssymb,amsthm,amsmath,amssymb,mathrsfs,enumerate}
\usepackage[matrix,arrow]{xy}
\usepackage[margin=2cm]{geometry}

\newtheorem{theorem}{Theorem}

\newtheorem{lemma}[theorem]{Lemma}
\newtheorem{corollary}[theorem]{Corollary}

\newtheorem*{maintheorem*}{Main Theorem}

\theoremstyle{definition}
\newtheorem{example}[theorem]{Example}

\theoremstyle{remark}
\newtheorem{remark}[theorem]{Remark}

\title{K-stable divisors in $\mathbb{P}^1\times \mathbb{P}^1\times\mathbb{P}^2$ of degree $(1,1,2)$}

\author{Ivan Cheltsov, Kento Fujita, Takashi Kishimoto, Takuzo Okada}

\thanks{Throughout this paper, all varieties are assumed to be projective and defined over~$\mathbb{C}$.}

\address{\emph{Ivan Cheltsov}
\newline
\textnormal{University of Edinburgh,  Edinburgh, Scotland}
\newline
\textnormal{\texttt{I.Cheltsov@ed.ac.uk}}}

\address{\emph{Kento Fujita}
\newline
\textnormal{Osaka University, Osaka, Japan}
\newline
\textnormal{\texttt{fujita@math.sci.osaka-u.ac.jp}}}

\address{\emph{Takashi Kishimoto}
\newline
\textnormal{Saitama University, Saitama, Japan}
\newline
\textnormal{\texttt{kisimoto.takasi@gmail.com}}}

\address{\emph{Takuzo Okada}
\newline
\textnormal{Saga University, Saga, Japan}
\newline
\textnormal{\texttt{okada@cc.saga-u.ac.jp}}}

\makeatletter\@addtoreset{equation}{section} \makeatother

\begin{document}
\maketitle


\begin{abstract}
We prove that every smooth divisor in $\mathbb{P}^1\times \mathbb{P}^1\times\mathbb{P}^2$ of degree $(1,1,2)$ is K-stable.
\end{abstract}

\tableofcontents

\section{Introduction}
\label{section:intro}

Smooth Fano threefolds have been classified by Iskovskikh, Mori and Mukai into $105$ families,
which are labeled as \textnumero 1.1, \textnumero 1.2,  \textnumero 1.3, $\ldots$, \textnumero 10.1. See \cite{ACCFKMGSSV} for the description of these families.
Threefolds in each of these $105$ deformation families can be parametrized by a~non-empty rational irreducible variety.
It has been proved in~\cite{ACCFKMGSSV,Fujita2019Crelle,Fujita2021} that the deformation families
\begin{center}
\textnumero 2.23, \textnumero 2.26, \textnumero 2.28, \textnumero 2.30, \textnumero 2.31, \textnumero 2.33, \textnumero 2.35, \textnumero 2.36, \textnumero 3.14, \\
\textnumero 3.16, \textnumero 3.18, \textnumero 3.21, \textnumero 3.22, \textnumero 3.23, \textnumero 3.24, \textnumero 3.26, \textnumero 3.28, \textnumero 3.29, \\
\textnumero 3.30, \textnumero 3.31, \textnumero 4.5,  \textnumero 4.8,  \textnumero 4.9,  \textnumero 4.10, \textnumero 4.11, \textnumero 4.12, \textnumero 5.2
\end{center}
do not have smooth K-polystable members,
and general members of the remaining 78 deformation families are K-polystable.
In fact, for 54 among these 78 families, we know all K-polystable smooth members \cite{AbbanZhuangSeshadri,ACCFKMGSSV,BelousovLoginov,CheltsovDenisovaFujita,CheltsovPark,Denisova,Liu,XuLiu}.
The remaining $24$ deformation families are
\begin{center}
\textnumero 1.9, \textnumero 1.10, \textnumero 2.5, \textnumero 2.9, \textnumero 2.10, \textnumero 2.11, \textnumero 2.12, \textnumero 2.13, \\
\textnumero 2.14, \textnumero 2.15, \textnumero 2.16, \textnumero 2.17, \textnumero 2.18, \textnumero 2.19, \textnumero 2.20, \textnumero 2.21, \\
\textnumero 3.2, \textnumero 3.3, \textnumero 3.4, \textnumero 3.5,    \textnumero 3.6, \textnumero 3.7, \textnumero 3.8, \textnumero 3.11.
\end{center}

The goal of this paper is to show that all smooth Fano threefolds in the family \textnumero 3.3 are \mbox{K-stable}.
Smooth members of this deformation family are smooth divisors in $\mathbb{P}^1\times \mathbb{P}^1\times\mathbb{P}^2$ of degree $(1,1,2)$.
To be precise, we prove the following result:

\begin{maintheorem*}
\label{theorem:main}
Let $X$ be a smooth divisor in $\mathbb{P}^1\times \mathbb{P}^1\times\mathbb{P}^2$ of degree $(1,1,2)$. Then $X$ is K-stable.
\end{maintheorem*}

We thank the~Nemuro city council and Saitama University for excellent working conditions.

Cheltsov has been supported by JSPS  Invitational Fellowships for Research in Japan (S22040) and by EPSRC Grant Number EP/V054597/1 (The Calabi problem for smooth Fano threefolds).
Fujita, Kishimoto and Okada have been supported by JSPS KAKENHI Grants Number 22K03269, 19K03395, JP22H01118, respectively.

We would like to thank an anonymous referee for useful comments.

\section{Smooth Fano threefolds in the~deformation family \textnumero 3.3}
\label{section:3-3}

Let $X$ be a  divisor in $\mathbb{P}^1_{s,t}\times \mathbb{P}^1_{u,v}\times\mathbb{P}^2_{x,y,z}$ of tridegree $(1,1,2)$,
where $([s:t],[u:v],[x:y:z])$ are coordinates on $\mathbb{P}^1_{s,t}\times \mathbb{P}^1_{u,v}\times\mathbb{P}^2_{x,y,z}$.
Then $X$ is given by the~following equation:
$$
\left[
  \begin{array}{cc}
    s & t\\
  \end{array}
\right]
\left[
  \begin{array}{cc}
    a_{11} & a_{12}\\
    a_{21} & a_{22}\\
  \end{array}
\right]
\left[
  \begin{array}{c}
    u \\
    v \\
  \end{array}
\right]=0,
$$
where each $a_{ij}=a_{ij}(x,y,z)$ is a homogeneous polynomials of degree $2$.
We can also define $X$ by
$$
\left[
  \begin{array}{ccc}
    x & y & z\\
  \end{array}
\right]
\left[
  \begin{array}{ccc}
    b_{11} & b_{12} & b_{13}\\
    b_{21} & b_{22} & b_{23}\\
    b_{31} & b_{32} & b_{33}\\
  \end{array}
\right]
\left[
  \begin{array}{c}
    x \\
    y \\
    z \\
  \end{array}
\right]=0,
$$
where each $b_{ij}=b_{ij}(s,t;u,v)$ is a bi-homogeneous polynomial of degree $(1,1)$.

Suppose that $X$ is smooth.
Then $X$ is a smooth Fano threefold in the~deformation family \textnumero 3.3.
Moreover, every smooth Fano threefold in this deformation family can be obtained in this way.
Observe that $-K_X^3=18$, and we have the~following commutative diagram:
$$
\xymatrix{
&&\mathbb{P}^1_{s,t}\times\mathbb{P}^1_{u,v}\ar@{->}[dll]\ar@{->}[drr]&&\\%
\mathbb{P}^1_{s,t}&& && \mathbb{P}^1_{u,v}\\
&&X\ar@{->}[dd]^{\pi_3}\ar@{->}[uu]^{\omega}\ar@{->}[rru]^{\pi_2}\ar@{->}[llu]_{\pi_1}\ar@{->}[dll]_{\phi_1}\ar@{->}[drr]^{\phi_2} &&\\
\mathbb{P}^1_{s,t}\times\mathbb{P}^2_{x,y,z}\ar@{->}[uu]\ar@{->}[rrd]&& && \mathbb{P}^1_{u,v}\times\mathbb{P}^2_{x,y,z}\ar@{->}[uu]\ar@{->}[lld]\\
&&\mathbb{P}^2_{x,y,z} &&}
$$
where all maps are induced by natural projections.
Note that $\omega$ is a~(standard) conic bundle whose discriminant curve $\Delta_{\mathbb{P}^1\times\mathbb{P}^1}\subset\mathbb{P}^1_{s,t}\times\mathbb{P}^1_{u,v}$ is a~(possibly singular) curve of degree $(3,3)$ given by
$$
\mathrm{det}\left[\begin{array}{ccc}
    b_{11} & b_{12} & b_{13}\\
    b_{21} & b_{22} & b_{23}\\
    b_{31} & b_{32} & b_{33}\\
  \end{array}
\right]=0.
$$
Similarly, the~map $\pi_3$ is a~(non-standard) conic bundle whose discriminant curve $\Delta_{\mathbb{P}^2}$ is a~smooth plane quartic curve in $\mathbb{P}^2_{x,y,z}$,
which is given by $a_{11}a_{22}=a_{12}a_{21}$.
Both maps $\phi_1$ and $\phi_2$ are birational morphisms that blow up the~following smooth genus $3$ curves:
\begin{align*}
\big\{sa_{11}+ta_{21}=sa_{12}+ta_{22}=0\big\}&\subset\mathbb{P}^1_{s,t}\times\mathbb{P}^2_{x,y,z},\\
\big\{ua_{11}+va_{12}=ua_{21}+va_{22}=0\big\}&\subset\mathbb{P}^1_{u,v}\times\mathbb{P}^2_{x,y,z}.
\end{align*}
Finally, both morphisms $\pi_1$ and $\pi_2$ are fibrations into quintic del Pezzo surfaces.

Let $H_1=\pi_1^*(\mathcal{O}_{\mathbb{P}^1}(1))$, let $H_2=\pi_2^*(\mathcal{O}_{\mathbb{P}^1}(1))$, let $H_3=\pi_3^*(\mathcal{O}_{\mathbb{P}^2}(1))$,
let $E_1$ and $E_2$ be the~exceptional divisors of the~morphisms $\phi_1$ and $\phi_2$, respectively.
Then
\begin{align*}
-K_X&\sim H_1+H_2+H_3,\\
E_1&\sim H_1+2H_3-H_2,\\
E_2&\sim H_2+2H_3-H_1.
\end{align*}
This gives $E_1+E_2\sim 4H_3$, which also follows from $E_1+E_2=\pi_3^*(\Delta_{\mathbb{P}^2})$. We have
$$
-K_X\sim_{\mathbb{Q}} \frac{3}{2}H_1+\frac{1}{2}H_2+\frac{1}{2}E_2\sim_{\mathbb{Q}} \frac{1}{2}H_1+\frac{3}{2}H_2+\frac{1}{2}E_1.
$$
In particular, we see that $\alpha(X)\leqslant\frac{2}{3}$.
Note that $E_1\cong E_2\cong \Delta_{\mathbb{P}^2}\times\mathbb{P}^1$.

The Mori cone $\overline{\mathrm{NE}}(X)$ is simplicial and is generated by the~curves contracted by $\omega$, $\phi_1$ and $\phi_2$.
The~cone of effective divisors $\mathrm{Eff}(X)$ is generated by the~classes of the~divisors $E_1$, $E_2$, $H_1$, $H_2$.

\begin{lemma}
\label{lemma:singular-quintic-del-Pezzo-surfaces}
Let $S$ be a surface in the~pencil $|H_1|$.
Then $S$ is a normal quintic del Pezzo surface that has at most Du Val singularities,
the restriction $\pi_3\vert_{S}\colon S\to\mathbb{P}^2_{x,y,z}$ is a birational morphism,
and the~restriction $\pi_2\vert_{S}\colon S\to\mathbb{P}^1_{u,v}$ is a conic bundle.
Moreover, one of the~following cases hold:
\begin{itemize}
\item  the~surface $S$ is smooth,

\item[($\mathbb{A}_1$)] the~surface $S$ has one singular point of type $\mathbb{A}_1$,

\item[($2\mathbb{A}_1$)] the~surface $S$ has two singular points of type $\mathbb{A}_1$,

\item[($\mathbb{A}_2$)] the~surface $S$ has one singular point of type $\mathbb{A}_2$,

\item[($\mathbb{A}_3$)] the~surface $S$ has one singular point of type $\mathbb{A}_3$.
\end{itemize}
Furthermore, in each of these five case, the~del Pezzo surface $S$ is unique up to an~isomorphism.
\end{lemma}

\begin{proof}
This is well-known \cite{CheltsovProkhorov,Coray1988}.
\end{proof}

\begin{remark}
\label{remark:singular-quintic-del-Pezzo-surfaces}
In the~notations and assumptions of Lemma~\ref{lemma:singular-quintic-del-Pezzo-surfaces},
suppose that the~surface $S$ is singular,
and let $\varpi\colon\widetilde{S}\to S$ be its minimal resolution of singularities.
Then the~dual graph of the~$(-1)$-curves and $(-2)$-curves on the~surface $\widetilde{S}$ can be described as follows:
\begin{itemize}
\item[($\mathbb{A}_1$)] if $S$ has one singular point of type $\mathbb{A}_1$, then the~dual graph is
$$
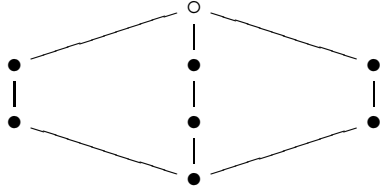
\xymatrix@R=0.8em{
&&\circ\ar@{-}[lld]\ar@{-}[d]\ar@{-}[rrd]&&\\
\bullet\ar@{-}[d]&&\bullet\ar@{-}[d]&&\bullet\ar@{-}[d]\\
\bullet&&\bullet&&\bullet\\
&&\bullet\ar@{-}[llu]\ar@{-}[u]\ar@{-}[rru]&&}
$$

\item[($2\mathbb{A}_1$)] if  $S$ has two singular points of type $\mathbb{A}_1$, then the~dual graph is
$$
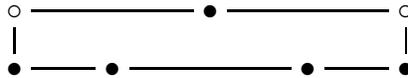
\xymatrix@R=0.8em{
\circ\ar@{-}[d]\ar@{-}[rr]&&\bullet\ar@{-}[rr]&&\circ\ar@{-}[d]\\
\bullet\ar@{-}[r]&\bullet\ar@{-}[rr]&&\bullet\ar@{-}[r]&\bullet}
$$

\item[($\mathbb{A}_2$)] if  $S$ has one singular point of type $\mathbb{A}_2$, then the~dual graph is
$$
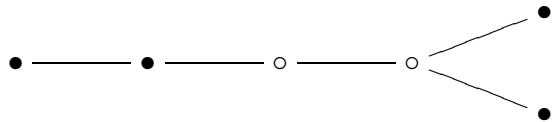
\xymatrix@R=7pt@C=37pt{
&&&&\bullet\\
\bullet\ar@{-}[r]&\bullet\ar@{-}[r]&\circ\ar@{-}[r]&\circ\ar@{-}[ru]\ar@{-}[rd]\\
&&&&\bullet}
$$

\item[($\mathbb{A}_3$)] if  $S$ has one singular point of type $\mathbb{A}_3$, then the~dual graph is
$$
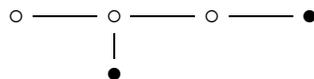
\xymatrix@R=0.8em{
\circ\ar@{-}[r]&\circ\ar@{-}[r]\ar@{-}[d]&\circ\ar@{-}[r]&\bullet\\
&\bullet}
$$
\end{itemize}
Here, as~in~the~papers \cite{Coray1988,CheltsovProkhorov}, we denote a~$(-1)$-curve by $\bullet$, and we denote a~$(-2)$-curve by $\circ$.
\end{remark}

\begin{lemma}
\label{lemma:S1-S2-singular}
Let $S_1$ be a surface in $|H_1|$, let $S_2$ be a surface in $|H_2|$, and let $P$ be a point in $S_1\cap S_2$.
Then at least one of the~surfaces $S_1$ or $S_2$ is smooth at $P$.
\end{lemma}

\begin{proof}
Local computations.
\end{proof}

\begin{corollary}
\label{corollary:S1-S2-singular}
In the~notations and assumptions of Lemma~\ref{lemma:S1-S2-singular},
suppose the~conic $S_1\cdot S_2$ is reduced.
Then at least one of the~surfaces $S_1$ or $S_2$ is smooth along $S_1\cap S_2$.
\end{corollary}

\begin{lemma}
\label{lemma:conic-conic-singular}
Let $P$ be a point in $X$, let $C$ be the~scheme fiber of the~conic bundle $\omega$ that contains~$P$,
and let $Z$ be the~scheme fiber of the~conic bundle $\pi_3$ that contains $P$.
Then $C$ or $Z$ is smooth at~$P$.
\end{lemma}

\begin{proof}
Local computations.
\end{proof}

\begin{lemma}
\label{lemma:quartic-weak-del-Pezzo-surfaces}
Let $C$ be a fiber of the~morphism $\pi_3$,
let $S$ be a general surface in $|H_3|$ that contains~$C$.
Then $S$ is smooth,  $K_S^2=4$ and $-K_S\sim (H_1+H_2)\vert_{S}$, which implies that $-K_S$ is nef and big.
Moreover, one of the~following three cases holds:
\begin{itemize}
\item[($\mathrm{1}$)]  the~conic $C$ is smooth, $-K_S$ is ample,
and the~restriction $\omega\vert_{S}\colon S\to\mathbb{P}^1_{s,t}\times\mathbb{P}^1_{u,v}$ is a double cover
branched over a smooth curve of degree $(2,2)$,

\item[($\mathrm{2}$)]  the~conic $C$ is smooth, the~divisor $-K_S$ is not ample, the~conic $\omega(C)$ is an irreducible component of the~discriminant curve $\Delta_{\mathbb{P}^1\times\mathbb{P}^1}$,
the conic $C$ is contained in $\mathrm{Sing}(\omega^{-1}(\Delta_{\mathbb{P}^1\times\mathbb{P}^1}))$,
and the~restriction map $\omega\vert_{S}\colon S\to\mathbb{P}^1_{s,t}\times\mathbb{P}^1_{u,v}$ fits the~following commutative diagram:
$$
\xymatrix{
&S\ar@{->}[dl]_\alpha\ar@{->}[dr]^{\omega\vert_{S}}&\\
\overline{S}\ar@{->}[rr]_\beta&&\mathbb{P}^1_{s,t}\times\mathbb{P}^1_{u,v}&&}
$$
where $\alpha$ is a birational morphism that contracts two disjoint $(-2)$-curves,
and $\beta$ is a double cover branched over a singular curve of degree $(2,2)$,
which is a union of the~curve $\omega(C)$ and another smooth curve of degree $(1,1)$,
which intersect transversally at two distinct points,

\item[($\mathrm{3}$)]  the~conic $C$ is singular, $-K_S$ is ample,
and the~restriction $\omega\vert_{S}\colon S\to\mathbb{P}^1_{s,t}\times\mathbb{P}^1_{u,v}$ is a double cover
branched over a smooth curve of degree $(2,2)$.
\end{itemize}
\end{lemma}

\begin{proof}
The smoothness of the surface $S$ easily follows from local computations.
If $-K_S$ is ample, the remaining assertions are obvious.
So, to complete the~proof, we assume that $-K_S$ is not ample.
Then the~restriction $\omega\vert_{S}\colon S\to\mathbb{P}^1_{s,t}\times\mathbb{P}^1_{u,v}$ fits the~ commutative diagram
$$
\xymatrix{
&S\ar@{->}[dl]_\alpha\ar@{->}[dr]^{\omega\vert_{S}}&\\
\overline{S}\ar@{->}[rr]_\beta&&\mathbb{P}^1_{s,t}\times\mathbb{P}^1_{u,v}&&}
$$
where $\alpha$ is a birational morphism that contracts all $(-2)$-curves in $S$,
and $\beta$ is a double cover branched over a singular curve of degree $(2,2)$.
Let $\ell$ be a $(-2)$-curve in $S$. Then
$$
(H_1+H_2)\cdot\ell=-K_S\cdot\ell=0,
$$
so that $\omega(\ell)$ is a point in $\mathbb{P}^1_{s,t}\times\mathbb{P}^1_{u,v}$.
But $\pi_3(\ell)$ is a line in $\mathbb{P}^2_{x,y,z}$ that contains the~point~$\pi_3(C)$.
This~shows that the~curve $\ell$ is an irreducible component of a singular fiber of the~conic bundle~$\omega$.
Therefore, we see that $\omega(\ell)\in \Delta_{\mathbb{P}^1\times\mathbb{P}^1}$.
This implies that the~conic bundle $\omega$ maps an~irreducible component of the~conic $C$ to an~irreducible component of the~curve~$\Delta_{\mathbb{P}^1\times\mathbb{P}^1}$,
because $S$ is a general surface in the~linear system $|H_3|$ that contains the~curve $C$.

If $C$ is singular, an irreducible component of the~curve $\Delta_{\mathbb{P}^1\times\mathbb{P}^1}$
is a curve of degree $(1,0)$ or $(0,1)$,
which is impossible~\mbox{\cite[\S~3.8]{Prokhorov}}.
Therefore, we see that the~conic $C$ is smooth and irreducible,
and the~curve $\omega(C)\cong C$ is an~irreducible component of the~discriminant curve~$\Delta_{\mathbb{P}^1\times\mathbb{P}^1}$.
Since the~conic bundle $\omega$ is standard \cite{Prokhorov},
the surface $\omega^{-1}(\omega(C))$ is irreducible and non-normal,
which easily implies that the~conic $C$ is contained in its singular locus.

Choosing appropriate coordinates on $\mathbb{P}^1_{s,t}\times \mathbb{P}^1_{u,v}\times\mathbb{P}^2_{x,y,z}$,
we may assume that $\pi_3(C)=[0:0:1]$, the~conic $C$ is given by $x=y=sv-tu=0$,
$([0:1],[0:1])$ is a smooth point of the~curve $\Delta_{\mathbb{P}^1\times\mathbb{P}^1}$,
and the~fiber $\omega^{-1}([0:1],[0:1])$ is given by $s=u=xy=0$.
Then $X$ is given by
\begin{multline*}
(a_1su+b_1sv+c_1tu)x^2+(a_2su+b_2sv+c_2tu+tv)xy+\\
+b_4(sv-tu)xz+(a_3su+b_3sv+c_3tu)y^2+b_5(sv-tu)yz+(sv-tu)z^2=0
\end{multline*}
for some numbers $a_1$, $a_2$, $a_3$, $b_1$, $b_2$, $b_3$, $b_4$, $b_5$, $c_1$,  $c_2$,   $c_3$.
One can check that $\Delta_{\mathbb{P}^1\times\mathbb{P}^1}$ indeed splits as a~union of the~curve $\omega(C)$~and
the curve in $\mathbb{P}^1_{s,t}\times \mathbb{P}^1_{u,v}$ of degree $(2,2)$ that is given by
\begin{multline*}
a_1b_5^2stu^2-a_1b_5^2s^2uv+a_2b_4b_5s^2uv-a_2b_4b_5stu^2-a_3b_4^2s^2uv+a_3b_4^2stu^2-b_1b_5^2s^2v^2+\\
+b_1b_5^2stuv+b_2b_4b_5s^2v^2-b_2b_4b_5stuv-b_3b_4^2s^2v^2+b_3b_4^2stuv-b_4^2c_3stuv+b_4^2c_3t^2u^2+b_4b_5c_2stuv-\\
-b_4b_5c_2t^2u^2-b_5^2c_1stuv+b_5^2c_1t^2u^2+4a_1a_3s^2u^2+4a_1b_3s^2uv+4a_1c_3stu^2-a_2^2s^2u^2-2a_2b_2s^2uv-\\
-2a_2c_2stu^2+4a_3b_1s^2uv+4a_3c_1stu^2+ 4b_1b_3s^2v^2+4b_1c_3stuv-b_2^2s^2v^2-2b_2c_2stuv+4b_3c_1stuv+\\
+b_4b_5stv^2-b_4b_5t^2uv+4c_1c_3t^2u^2-c_2^2t^2u^2-2a_2stuv-2b_2stv^2-2c_2t^2uv-t^2v^2=0.
\end{multline*}

The surface $S$ is cut out on $X$ by the~equation $y=\lambda x$, where $\lambda$ is a general complex number.
Then the~double cover $\beta\colon\overline{S}\to\mathbb{P}^1_{s,t}\times \mathbb{P}^1_{u,v}$
is branched over a singular curve of degree $(2,2)$, which splits as a union of the~curve
$\omega(C)$~and the~curve in $\mathbb{P}^1_{s,t}\times \mathbb{P}^1_{u,v}$ of degree $(1,1)$ that is given by
\begin{multline*}
\lambda^2 b_5^2tu-\lambda^2b_5^2sv+4\lambda^2a_3su+4\lambda^2b_3sv-2b_4\lambda b_5sv+2\lambda b_4b_5tu+\\
+4\lambda^2c_3tu+4\lambda a_2su+4\lambda b_2sv-b_4^2sv+b_4^2tu+4\lambda c_2tu+4a_1su+4b_1sv+4c_1tu+4\lambda tv=0.
\end{multline*}
Since $\lambda$ is general and $X$ is smooth, these two curves intersect transversally by two points,
which implies the~remaining assertions of the~lemma.
\end{proof}

Note that the case ($\mathrm{2}$) in Lemma~\ref{lemma:quartic-weak-del-Pezzo-surfaces} indeed can happen.
For instance, if $X$ is given by
$$
(sv+tu)x^2+(su-sv+tv)xy+(5sv-5tu)zx+3suy^2+(sv-tu)zy+(sv-tu)z^2=0,
$$
then $X$ is smooth, and general surface in $|H_3|$ that contains the~curve $\pi_3^{-1}([0:0:1])$ is a~smooth weak del Pezzo surface, which is not a~quartic del Pezzo surface.

\begin{lemma}
\label{lemma:del-Pezzo-surface-degree-two}
Let $C$ be a smooth fiber of the~morphism $\omega$,
and let $S$ be a general surface in $|H_1+H_2|$ that contains~the~curve $C$.
Then $S$ is a smooth del Pezzo surface of degree $2$, and $-K_S\sim H_3\vert_{S}$.
\end{lemma}

\begin{proof}
Left to the~reader.
\end{proof}

\section{Applications of Abban--Zhuang theory}
\label{section:Abban-Zhuang}

Let us use notations and assumptions of Section~\ref{section:3-3}.
Let $f\colon\widetilde{X}\to X$ be a birational map such that $\widetilde{X}$ is a normal threefold,
and let $\mathbf{F}$ be a prime divisor in $\widetilde{X}$.
Then, to prove that $X$ is K-stable, it is enough to show that $\beta(\mathbf{F})=A_X(\mathbf{F})-S_X(\mathbf{F})>0$,
where $A_X(\mathbf{F})=1+\mathrm{ord}_{\mathbf{F}}(K_{\widetilde{X}}/K_X)$ and
$$
S_X(\mathbf{F})=\frac{1}{-K_X^3}\int_{0}^{\infty}\mathrm{vol}\big(f^*(-K_X)-u\mathbf{F}\big)du.
$$
This follows from the~valuative criterion for K-stability \cite{Fujita2019Crelle,Li}.

Let $\mathfrak{C}$ be the~center of the~divisor $\mathbf{F}$ on the~threefold $X$.
By \cite[Theorem~10.1]{Fujita2016}, we have
$$
S_X(S)=\frac{1}{-K_X^3}\int_{0}^{\infty}\mathrm{vol}\big(-K_X-uS\big)du<1
$$
for every surface $S\subset X$. Hence, if $\mathfrak{C}$ is a surface, then $\beta(\mathbf{F})>0$ .
Thus, to show that $X$ is~K-stable, we may assume that $\mathfrak{C}$ is either a~curve or a~point.
If $\mathfrak{C}$ is a curve, then \cite[Corollary~1.7.26]{ACCFKMGSSV}~gives

\begin{corollary}
\label{corollary:curve}
Suppose that $\beta(\mathbf{F})\leqslant 0$ and $\mathfrak{C}$ is a curve.
Let $S$ be an irreducible normal surface in the~threefold $X$ that contains~$\mathfrak{C}$.~
Set
\begin{multline*}
\quad \quad \quad \quad S\big(W^S_{\bullet,\bullet};\mathfrak{C}\big)=\frac{3}{(-K_X)^3}\int_0^\tau\big(P(u)^{2}\cdot S\big)\cdot\mathrm{ord}_{\mathfrak{C}}\big(N(u)\big\vert_{S}\big)du+\\
+\frac{3}{(-K_X)^3}\int_0^\tau\int_0^\infty \mathrm{vol}\big(P(u)\big\vert_{S}-v\mathfrak{C}\big)dvdu,\quad \quad \quad \quad
\end{multline*}
where $\tau$ is the~largest rational number $u$ such that  $-K_X-uS$ is pseudo-effective,
$P(u)$ is the~positive part of the~Zariski decomposition of $-K_X-uS$, and $N(u)$ is its negative part.
Then $S(W^S_{\bullet,\bullet};\mathfrak{C})>1$.
\end{corollary}

Let $P$ be a~point in $\mathfrak{C}$. Then
$$
\frac{A_X(\mathbf{F})}{S_X(\mathbf{F})}\geqslant\delta_P(X)=\inf_{\substack{E/X\\ P\in C_X(E)}}\frac{A_{X}(E)}{S_X(E)},
$$
where the~infimum is taken over all prime divisors $E$ over $X$ whose centers on $X$ that contain~$P$.
Therefore, to prove that the~Fano threefold $X$ is K-stable, it is enough to show that $\delta_P(X)>1$.
On~the other hand, we can estimate $\delta_P(X)$ by using \cite[Theorem~3.3]{AbbanZhuang} and \cite[Corollary~1.7.30]{ACCFKMGSSV}.
Namely,~let $S$~be~an~irreducible surface in $X$ with Du Val singularities such that $P\in S$.~Set
$$
\tau=\mathrm{sup}\Big\{u\in\mathbb{Q}_{\geqslant 0}\ \big\vert\ \text{the divisor  $-K_X-uS$ is pseudo-effective}\Big\}.
$$
For~$u\in[0,\tau]$, let $P(u)$ be the~positive part of the~Zariski decomposition of the~divisor $-K_X-uS$,
and let $N(u)$ be its negative part.
Then \cite[Theorem~3.3]{AbbanZhuang} and  \cite[Corollary~1.7.30]{ACCFKMGSSV} give
\begin{equation}
\label{equation:Hamid-Ziquan}
\delta_P(X)\geqslant\mathrm{min}\Bigg\{\frac{1}{S_X(S)},\delta_{P}\big(S;W^S_{\bullet,\bullet}\big)\Bigg\}
\end{equation}
for
$$
\delta_{P}\big(S;W^S_{\bullet,\bullet}\big)=\inf_{\substack{F/S,\\ P\subseteq C_S(F)}}\frac{A_S(F)}{S(W^S_{\bullet,\bullet};F)},
$$
where
$$
S\big(W^S_{\bullet,\bullet}; F\big)=\frac{3}{-K_X^3}\int_0^\tau\big(P(u)^{2}\cdot S\big)\cdot\mathrm{ord}_{F}\big(N(u)\big\vert_{S}\big)du+\frac{3}{-K_X^3}\int_{0}^{\tau}\int_0^\infty \mathrm{vol}\big(P(u)\big\vert_{S}-vF\big)dvdu,
$$
and now the~infimum is taken over all prime divisors $F$ over $S$ whose centers on $S$ that contain~$P$.
Let us show how to apply \eqref{equation:Hamid-Ziquan} in some cases. Recall that \mbox{$S_X(S)<1$} by \cite[Theorem~10.1]{Fujita2016}.

\begin{lemma}
\label{lemma:delta-P-del-Pezzo-degree-4-smooth}
Let $C$ be the~fiber of the~conic bundle $\pi_3$ that contains $P$,
and let $S$ be a general surface in $|H_3|$ that contains $C$.
Suppose $S$ is a smooth del Pezzo of degree~$4$, and $C$ is smooth.
Then~$\delta_P(X)>1$.
\end{lemma}

\begin{proof}
One has $\tau=1$.
Moreover, for $u\in[0,1]$, we have $N(u)=0$ and $P(u)|_S=-K_S+(1-u)C$.
Let $L=-K_S+(1-u)C$.
Using Lemma~\ref{lemma:delta-dp4} and arguing as in the~proof of Lemma~\ref{lemma:Nemuro},
we get
\begin{multline*}
\quad \quad  \quad  S\big(W^S_{\bullet,\bullet};F\big)=\frac{1}{6}\int_0^1 4(1+(1-u))S_L(F)du\leqslant \\
\leqslant A_S(F)\int_0^1 \frac{4}{6}(1+(1-u)) \frac{19+8(1-u)+(1-u)^2}{24}du=\frac{143}{144}A_S(F)\quad \quad
\end{multline*}
for any prime divisor $F$ over $S$ such that $P\in C_S(F)$.
Then \eqref{equation:Hamid-Ziquan} gives $\delta_P(X)>1$.
\end{proof}

Similarly, we obtain the~following result:

\begin{lemma}
\label{lemma:delta-P-del-Pezzo-smooth-A1}
Let $S$ be the~surface in $|H_1|$ that contain $P$. Then
$$
\delta_P(X)\geqslant\mathrm{min}\Bigg\{\frac{1}{S_X(S)},\frac{2592\delta_P(S)}{2560+63\delta_P(S)}\Bigg\}
$$
for $\delta_P(S)=\delta_P(S,-K_S)$, where $\delta_P(S,-K_S)$ is defined in Appendix~\ref{section:delta-dP}.
\end{lemma}

\begin{proof}
We have $\tau=\frac{3}{2}$. Moreover, we have
$$
P(u)=\left\{\aligned
&(1-u)H_1+H_2+H_3\ \text{if $0\leqslant u\leqslant 1$}, \\
&(2-u)H_2+(3-2u)H_3\ \text{if $1\leqslant u\leqslant \frac{3}{2}$}, \\
\endaligned
\right.
$$
and
$$
N(u)=\left\{\aligned
&0\ \text{if $0\leqslant u\leqslant 1$}, \\
&(u-1)E_2\ \text{if $1\leqslant u\leqslant \frac{3}{2}$}.\\
\endaligned
\right.
$$
Note also that $E_2\vert_{S}$ is a smooth genus $3$ curve contained in the~smooth locus of the~surface $S$.

Recall that $S$ is a quintic del Pezzo surface with at most Du Val singularities,
and the~restriction morphism $\pi_2\vert_{S}\colon S\to\mathbb{P}^1_{u,v}$ is a conic bundle.
Note that the~morphism $\pi_3\vert_{S}\colon S\to\mathbb{P}^2_{x,y,z}$ is birational.
Let~$C$ be a fiber of the~conic bundle $\pi_2\vert_{S}$,
and let $L$ be the~preimage in $S$ of a general line in $\mathbb{P}^2_{x,y,z}$.
Then $-K_S\sim C+L$ and
$$
P(u)\big\vert_{S}\sim_{\mathbb{R}}\left\{\aligned
&C+L\ \text{if $0\leqslant u\leqslant 1$}, \\
&(2-u)C+(3-2u)L\ \text{if $1\leqslant u\leqslant \frac{3}{2}$}, \\
\endaligned
\right.
$$
Since $2L-C$ is pseudoeffective, the~divisor $\frac{7-4u}{3}(-K_S)-(2-u)C-(3-2u)L$ is also pseudoeffective.

Let $F$ be a divisor over $S$ such that $P\in C_S(F)$.
Then it follows from Lemma~\ref{lemma:Nemuro} that
\begin{multline*}
S\big(W^S_{\bullet,\bullet};F\big)\leqslant\frac{1}{6}A_S(F)\int_1^\frac{3}{2}(u-1)\big(P(u)\big\vert_{S}\big)^2du+\frac{1}{6}\int_{0}^{\frac{3}{2}}\int_0^\infty \mathrm{vol}\big(P(u)\big\vert_{S}-vF\big)dvdu=\\
=\frac{7}{288}A_S(F)+\frac{1}{6}\int_{0}^{1}\int_0^\infty \mathrm{vol}\big(-K_S-vF\big)dvdu+\frac{1}{6}\int_{1}^{\frac{3}{2}}\int_0^\infty\mathrm{vol}\big((2-u)C+(3-2u)L-vF\big)dvdu\leqslant\\
\leqslant\frac{7}{288}A_S(F)+\frac{1}{6}\int_{0}^{1}5\frac{A_S(F)}{\delta_P(S)}du+\frac{1}{6}\int_{1}^{\frac{3}{2}}\int_0^\infty\mathrm{vol}\Bigg(\frac{7-4u}{3}\big(-K_S\big)-vF\Bigg)dvdu=\\
=\frac{7}{288}A_S(F)+\frac{5}{6\delta_P(S)}A_S(F)+\frac{1}{6}\int_{1}^{\frac{3}{2}}\Bigg(\frac{7-4u}{3}\Bigg)^3\int_0^\infty\mathrm{vol}\big(-K_S-vF\big)dvdu\leqslant\\
\leqslant\frac{7}{288}A_S(F)+\frac{5}{6\delta_P(S)}A_S(F)+\frac{1}{6}\int_{1}^{\frac{3}{2}}\Bigg(\frac{7-4u}{3}\Bigg)^35\frac{A_S(F)}{\delta_P(S)}du=\\
=\frac{7}{288}A_S(F)+\frac{5}{6\delta_P(S)}A_S(F)+\frac{25}{162\delta_P(S)}A_S(F)=\Bigg(\frac{80}{81\delta_P(S)}+\frac{7}{288}\Bigg)A_S(F),
\end{multline*}
Then $\delta_{P}(S;W^S_{\bullet,\bullet})\geqslant\frac{1}{\frac{80}{81\delta_P(S)}+\frac{7}{288}}=\frac{2592\delta_P(S)}{2560+63\delta_P(S)}$
and the~required assertion follows from \eqref{equation:Hamid-Ziquan}.
\end{proof}

Keeping in mind that $S_X(S)<1$ by \cite[Theorem~10.1]{Fujita2016} and the~$\delta$-invariant of the~smooth quintic del Pezzo surface is $\frac{15}{13}$ by \cite[Lemma 2.11]{ACCFKMGSSV}, we obtain

\begin{corollary}
\label{corollary:delta-P-quintic-del-Pezzo-smooth}
Let $S$ be the~surface in $|H_1|$ that contain $P$.
If $S$ is smooth, then $\delta_P(X)>1$.
\end{corollary}

Similarly, using Lemmas~\ref{lemma:delta-dP-5-A1} and \ref{lemma:delta-dP-5-A1-A1} from Appendix~\ref{section:delta-dP}, we obtain

\begin{corollary}
\label{corollary:delta-P-quintic-del-Pezzo-A1}
Let $S$ be the~surface in $|H_1|$ that contain $P$.
Suppose that $S$ has at most singular points of type $\mathbb{A}_1$,
and $P$ is not contained in any line in $S$ that passes through a singular point.
Then $\delta_P(X)>1$.
\end{corollary}

Alternatively, we can estimate $\delta_P(X)$ using \cite[Theorem~1.7.30]{ACCFKMGSSV}.
Namely, let $C$ be an irreducible smooth curve in $S$ that contains $P$.
Suppose $S$ is smooth at $P$.
Since $S\not\subset\mathrm{Supp}(N(u))$,  we write
$$
N(u)\big\vert_S=d(u)C+N_S^\prime(u),
$$
where $N_S^\prime(u)$ is an~effective $\mathbb{R}$-divisor on $S$ such that $C\not\subset\mathrm{Supp}(N_S^\prime(u))$, and \mbox{$d(u)=\mathrm{ord}_C(N(u)\vert_S)$}.
Now, for every $u\in [0,\tau]$, we define the~pseudo-effective threshold $t(u)\in\mathbb{R}_{\geqslant 0}$ as follows:
$$
t(u)=\inf\Big\{v\in \mathbb R_{\geqslant 0} \ \big|\ \text{the divisor $P(u)\big|_S-vC$ is pseudo-effective}\Big\}.
$$
For $v\in [0, t(u)]$, we let $P(u,v)$ be the~positive part of the~Zariski decomposition of $P(u)|_S-vC$, and we let $N(u,v)$ be its negative part.
As in Corollary~\ref{corollary:curve}, we let
\begin{multline*}
\quad \quad \quad \quad S\big(W^S_{\bullet,\bullet};C\big)=\frac{3}{(-K_X)^3}\int_0^\tau\big(P(u)^{2}\cdot S\big)\cdot\mathrm{ord}_{C}\big(N(u)\big\vert_{S}\big)du+\\
+\frac{3}{(-K_X)^3}\int_0^\tau\int_0^\infty \mathrm{vol}\big(P(u)\big\vert_{S}-vC\big)dvdu.\quad \quad \quad \quad
\end{multline*}
Note that $C\not\subset \mathrm{Supp}(N(u,v))$ for every $u\in [0, \tau)$ and $v\in (0, t(u))$.
Thus, we can let
$$
F_P\big(W_{\bullet,\bullet,\bullet}^{S,C}\big)=\frac{6}{(-K_X)^3} \int_0^\tau\int_0^{t(u)}\big(P(u,v)\cdot C\big)\cdot \mathrm{ord}_P\big(N_S^\prime(u)\big|_C+N(u,v)\big|_C\big)dvdu.
$$
Finally, we let
$$
S\big(W_{\bullet, \bullet,\bullet}^{S,C};P\big)=\frac{3}{(-K_X)^3}\int_0^\tau\int_0^{t(u)}\big(P(u,v)\cdot C\big)^2dvdu+F_P\big(W_{\bullet,\bullet,\bullet}^{S,C}\big).
$$
Then \cite[Theorem~1.7.30]{ACCFKMGSSV} gives

\begin{corollary}
\label{corollary:AZ-twice}
One has
\begin{equation}
\label{equation:AZ-delta-inequality}\tag{$\bigstar$}
\frac{A_X(\mathbf{F})}{S_X(\mathbf{F})}\geqslant\delta_P(X)\geqslant \min\left\{\frac{1}{S(W_{\bullet, \bullet,\bullet}^{S,C}; P)}, \frac{1}{S(W_{\bullet,\bullet}^S;C)},\frac{1}{S_X(S)}\right\}.
\end{equation}
Moreover, if both inequalities in \eqref{equation:AZ-delta-inequality} are~equalities and $\mathfrak{C}=P$,
then $\delta_P(X)=\frac{1}{S_X(S)}$.
\end{corollary}

Let us show how to compute $S(W_{\bullet,\bullet}^S;C)$ and $S(W_{\bullet, \bullet,\bullet}^{S,C};P)$ in some cases.

\begin{lemma}
\label{lemma:S-C-P-point-not-in-Delta-P1-P1}
Suppose that $\omega(P)\not\in\Delta_{\mathbb{P}^1\times\mathbb{P}^1}$.
Let $S$ be a general surface in $|H_1+H_2|$ that contains~$P$,
and let $C$ be the~fiber of the~morphism $\omega$ containing $P$.
Then $S(W_{\bullet,\bullet}^S;C)=\frac{31}{36}$ and $S(W_{\bullet, \bullet,\bullet}^{S,C};P)=1$.
\end{lemma}

\begin{proof}
We have $\tau=1$.
Moreover, for $u\in[0,1]$, we have $N(u)=0$ and $P(u)|_S=-K_S+2(1-u)C$.
On the~other hand, it follows from Lemma~\ref{lemma:del-Pezzo-surface-degree-two}  that $S$ is a smooth del Pezzo surface of degree $2$,
and the~restriction map $\pi_3\vert_{S}\colon S\to\mathbb{P}^2_{x,y,z}$ is a double cover that is ramified over a smooth quartic curve.
Therefore, applying the~Galois involution of this double cover to $C$,
we obtain another smooth irreducible curve $Z\subset S$ such that $C+Z\sim -2K_S$, $C^2=Z^2=0$ and $C\cdot Z=4$, which gives
$$
P(u)|_S-vC\sim_{\mathbb{R}}\Big(\frac{5}{2}-2u-v\Big)C+\frac{1}{2}Z.
$$
Then $P(u)\vert_{S}-vC$ is pseudoeffective $\iff$ $P(u)\vert_{S}-vC$ is nef $\iff$ $v\leqslant \frac{5}{2}-2u$.
Thus, we have
$$
\mathrm{vol}\big(P(u)\vert_{S}-vC\big)=\big(-K_S+2(1-u)C\big)^2=10-8u-4v
$$
and $P(u,v)\cdot C=2$. Now, integrating, we get $S(W_{\bullet,\bullet}^S;C)=\frac{31}{36}$ and $S(W_{\bullet,\bullet,\bullet}^{S,C};P)=1$.
\end{proof}

\begin{lemma}
\label{lemma:S-C-P-point-not-in-E1-E2-1}
Suppose that $P\not\in E_1\cup E_2$.
Let $S$ be a general surface in $|H_3|$ that contains~$P$,
and let $C$ be the~fiber of the~morphism $\pi_3$ containing $P$.
Suppose that $S$ is a smooth del Pezzo surface.
Then $S(W_{\bullet,\bullet}^S;C)=\frac{7}{9}$ and $S(W_{\bullet, \bullet,\bullet}^{S,C};P)=1$.
\end{lemma}

\begin{proof}
We have $\tau=1$.
Moreover, for $u\in[0,1]$, we have $N(u)=0$ and $P(u)|_S=-K_S+(1-u)C$.
Since $S$ is a smooth del Pezzo surface,
the restriction map $\omega\vert_{S}\colon S\to\mathbb{P}^1_{s,t}\times \mathbb{P}^1_{u,v}$ is a double cover ramified over a smooth elliptic curve.
Therefore, using the~Galois involution of this double cover,
we get an irreducible curve $Z\subset S$ such that $C+Z\sim -K_S$, $C^2=Z^2=0$, $C\cdot Z=2$, which gives
$$
P(u)|_S-vC\sim_{\mathbb{R}}(2-u-v)C+Z.
$$
Then $P(u)\vert_{S}-vC$ is pseudoeffective $\iff$ $P(u)\vert_{S}-vC$ is nef $\iff$ $v\leqslant 2-u$.
Thus, we have
$$
\mathrm{vol}\big(P(u)\vert_{S}-vC\big)=\big(-K_S+(1-u)C\big)^2=8-4u-4v
$$
and $P(u,v)\cdot C=2$.
Now, integrating, we obtain $S(W_{\bullet,\bullet}^S;C)=\frac{7}{9}$ and $S(W_{\bullet,\bullet,\bullet}^{S,C};P)=1$.
\end{proof}

\begin{lemma}
\label{lemma:S-C-P-point-not-in-E1-E2-2}
Suppose that $P\not\in E_1\cup E_2$.
Let $S$ be a general surface in $|H_3|$ that contains~$P$,
and let $C$ be the~fiber of the~morphism $\pi_3$ containing $P$.
Suppose $S$ is not a smooth del Pezzo surface.
Then $S(W_{\bullet,\bullet}^S;C)=\frac{8}{9}$ and $S(W_{\bullet, \bullet,\bullet}^{S,C};P)=\frac{7}{9}$.
\end{lemma}

\begin{proof}
We have $\tau=1$.
Moreover, for $u\in[0,1]$, we have $N(u)=0$ and $P(u)|_S=-K_S+(1-u)C$.
It follows from Lemma~\ref{lemma:quartic-weak-del-Pezzo-surfaces} that
$S$ contains two $(-2)$-curves $\mathbf{e}_1$ and $\mathbf{e}_2$ such that
$-K_S\sim 2C+\mathbf{e}_1+\mathbf{e}_2$.
On the~surface $S$, we have $C^2=0$, $C\cdot \mathbf{e}_1=C\cdot \mathbf{e}_2=1$, $\mathbf{e}_1^2=\mathbf{e}_2^2=-2$, and
$$
P(u)|_S-vC\sim_{\mathbb{R}}(3-u-v)C+\mathbf{e}_1+\mathbf{e}_2.
$$
Then $P(u)\vert_{S}-vC$ is pseudoeffective $\iff$ $v\leqslant 3-u$.
Moreover, we have
$$
P(u,v)=\left\{\aligned
&(3-u-v)C+\mathbf{e}_1+\mathbf{e}_2\ \text{if $0\leqslant v\leqslant 1-u$}, \\
&\frac{3-u-v}{2}\big(2C+\mathbf{e}_1+\mathbf{e}_2\big)\ \text{if $1-u\leqslant v\leqslant 3-u$}, \\
\endaligned
\right.
$$
$$
N(u,v)=\left\{\aligned
&0\ \text{if $0\leqslant v\leqslant 1-u$}, \\
&\frac{u+v-1}{2}(\mathbf{e}_1+\mathbf{e}_2)\ \text{if $1-u\leqslant v\leqslant 3-u$}, \\
\endaligned
\right.
$$
$$
\mathrm{vol}\big(P(u)\vert_{S}-vC\big)=
\left\{\aligned
&8-4u-4v\ \text{if $0\leqslant v\leqslant 1-u$}, \\
&(u+v-3)^2\ \text{if $1-u\leqslant v\leqslant 3-u$}. \\
\endaligned
\right.
$$
Now, integrating $\mathrm{vol}(P(u)\vert_{S}-vC)$, we obtain $S(W_{\bullet,\bullet}^S;C)=\frac{8}{9}$.

To compute $S(W_{\bullet,\bullet,\bullet}^{S,C};P)$, observe that $F_P(W_{\bullet,\bullet,\bullet}^{S,C})=0$,
because $P\not\in \mathbf{e}_1\cup \mathbf{e}_2$, since $S$ is a~general surface in $|H_3|$ that contains  $C$.
On the~other hand, we have
$$
P(u,v)\cdot C=\left\{\aligned
&2\ \text{if $0\leqslant v\leqslant 1-u$}, \\
&3-u-v\ \text{if $1-u\leqslant v\leqslant 3-u$}. \\
\endaligned
\right.
$$
Hence, integrating $(P(u,v)\cdot C)^2$, we get $S(W_{\bullet,\bullet,\bullet}^{S,C};P)=\frac{7}{9}$ as required.
\end{proof}

\begin{lemma}
\label{lemma:S-C-P-point-not-in-E1-E2-3}
Suppose $P\in(E_1\cup E_2)\setminus(E_1\cap E_2)$.
Let $S$ be a general surface in $|H_3|$ that contains~$P$,
let $C$ be the~irreducible component of the~fiber of the~conic bundle $\pi_3$ containing $P$ such that $P\in C$.
Then $S(W_{\bullet,\bullet}^S;C)=1$ and $S(W_{\bullet, \bullet,\bullet}^{S,C};P)\leqslant \frac{31}{36}$.
\end{lemma}

\begin{proof}
We have $\tau=1$.
For $u\in[0,1]$, we have $N(u)=0$ and
$P(u)|_S\sim_{\mathbb{R}}-K_S+(1-u)(C+C^\prime)$,
where $C^\prime$ is the~irreducible curve in $S$ such that $C+C^\prime$ is the~fiber of the~conic bundle $\pi_3$ that passes through the~point $P$.
Since $P\not\in E_1\cap E_2$, we see that $P\not\in C^\prime$.

By Lemma~\ref{lemma:quartic-weak-del-Pezzo-surfaces}, the~surface $S$ is a smooth del Pezzo surface of degree $4$,
so we can identify it with a complete intersection of two quadrics in $\mathbb{P}^4$.
Then $C$ and $C^\prime$ are lines in $S$, and  $S$ contains four additional lines that intersect $C$.
Denote them by $L_1$, $L_2$, $L_3$, $L_4$, and let $Z=L_1+L_2+L_3+L_4$.
Then the~intersections of the~curves $C$, $C^\prime$ and $Z$ on the~surface $S$ are given in the~table below.
\begin{center}
\renewcommand\arraystretch{1.2}
\begin{tabular}{|c||c|c|c|}
\hline
$\bullet$  & $C$ & $C^\prime$ & $Z$ \\
\hline\hline
$C$        & $-1$ & $1$ & $4$ \\
\hline
$C^\prime$  & $1$ & $-1$ & $0$\\
\hline
$Z$ & $4$ & $0$ & $-4$ \\
\hline
\end{tabular}
\end{center}

Observe that $-K_S\sim_{\mathbb{Q}}\frac{3}{2}C+\frac{1}{2}C^\prime+\frac{1}{2}Z$. This gives
$P(u)\vert_{S}-vC\sim_{\mathbb{R}}(\frac{5}{2}-u-v)C+(\frac{3}{2}-u)C^\prime+\frac{1}{2}Z$,
which implies that $P(u)\vert_{S}-vC$ is pseudoeffective $\iff$ $v\leqslant \frac{5}{2}-u$.

Moreover, we have
$$
P(u,v)=\left\{\aligned
&\Big(\frac{5}{2}-u-v\Big)C+\Big(\frac{3}{2}-u\Big)C^\prime+\frac{1}{2}Z\ \text{if $0\leqslant v\leqslant 1$}, \\
&\Big(\frac{5}{2}-u-v\Big)(C+C^\prime)+\frac{1}{2}Z\ \text{if $1\leqslant v\leqslant 2-u$}, \\
&\Big(\frac{5}{2}-u-v\Big)(C+C^\prime+Z)\ \text{if $2-u\leqslant v\leqslant \frac{5}{2}-u$}, \\
\endaligned
\right.
$$
$$
N(u,v)=\left\{\aligned
&0\ \text{if $0\leqslant v\leqslant 1$}, \\
&(v-1)C^\prime\ \text{if $1\leqslant v\leqslant 2-u$}, \\
&(v-1)C^\prime+(v+u-2)Z\ \text{if $2-u\leqslant v\leqslant \frac{5}{2}-u$}, \\
\endaligned
\right.
$$
$$
P(u,v)\cdot C=\left\{\aligned
&1+v\ \text{if $0\leqslant v\leqslant 1$}, \\
&2\ \text{if $1\leqslant v\leqslant 2-u$}, \\
&10-4u-4v\ \text{if $2-u\leqslant v\leqslant \frac{5}{2}-u$},
\endaligned
\right.
$$
$$
\mathrm{vol}\big(P(u)\vert_{S}-vC\big)=
\left\{\aligned
&8-v^2-4u-2v\ \text{if $0\leqslant v\leqslant 1$}, \\
&9-4u-4v\ \text{if $1\leqslant v\leqslant 2-u$}, \\
&(5-2u-2v)^2\ \text{if $2-u\leqslant v\leqslant \frac{5}{2}-u$}. \\
\endaligned
\right.
$$
Now, integrating $\mathrm{vol}(P(u)\vert_{S}-vC)$ and $(P(u,v)\cdot C)^2$, we get $S(W_{\bullet,\bullet}^S;C)=1$ and
\begin{multline*}
S\big(W_{\bullet, \bullet,\bullet}^{S,C};P\big)=\frac{5}{6}+F_P\big(W_{\bullet,\bullet,\bullet}^{S,C}\big)=\frac{5}{6}+\frac{1}{3}\int_0^1\int_0^{\frac{5}{2}-u}\big(P(u,v)\cdot C\big)\cdot \mathrm{ord}_P\big(N(u,v)\big|_C\big)dvdu\leqslant\\
\leqslant\frac{5}{6}+\frac{1}{3}\int_0^1\int_2^{\frac{5}{2}-u}(10-4u-4v)(v+u-2)dvdu=\frac{31}{36},\quad\quad\quad
\end{multline*}
because $P\not\in C^\prime$, and the~curves $Z$ and $C$ intersect each other transversally.
\end{proof}

\section{The proof of Main Theorem}
\label{section:the-proof}

Let us use notations and assumptions of Sections~\ref{section:3-3} and \ref{section:Abban-Zhuang}.
Recall that $\mathbf{F}$ is a prime divisor  over the~threefold $X$, and $\mathfrak{C}$ is its center in $X$.
To prove Main Theorem, we must show that $\beta(\mathbf{F})>0$.

\begin{lemma}
\label{lemma:the-proof-curve}
Suppose that $\mathfrak{C}$ is a curve. Then $\beta(\mathbf{F})>0$.
\end{lemma}

\begin{proof}
Suppose $\beta(\mathbf{F})\leqslant 0$.
Then $\delta_P(X)\leqslant 1$ for every point $P\in \mathfrak{C}$.
Let us seek for a contradiction.

Let $S_1$ be a general surface in the~linear system $|H_1|$. Then $S_1$ is smooth.
Hence, if $S_1\cap \mathfrak{C}\ne\varnothing$, then  $\delta_P(X)\leqslant 1$ for every point $P\in S_1\cap \mathfrak{C}$,
which contradicts Corollary~\ref{corollary:delta-P-quintic-del-Pezzo-smooth}.
We see that $S_1\cdot \mathfrak{C}=0$.
Similarly, we see that $S_2\cdot \mathfrak{C}=0$ for a general surface $S_2\in|H_2|$.
So, we see that $\omega(\mathfrak{C})$ is a point.

Let $C$ be the~scheme fiber of the~conic bundle $\omega$ over the~point $\omega(\mathfrak{C})$.
Then $\mathfrak{C}$ is an irreducible component of the~curve $C$.
If the~fiber $C$ is smooth, then we $\mathfrak{C}=C$.

Suppose that $C$ is smooth.
If $S$ is a general surface in the~linear system $|H_1+H_2|$ that contains~$\mathfrak{C}$, then
$S(W_{\bullet,\bullet}^S;\mathfrak{C})=\frac{31}{36}<1$ by Lemma~\ref{lemma:S-C-P-point-not-in-Delta-P1-P1},
which contradicts Corollary~\ref{corollary:curve}.
So, the curve  $C$ is singular.

Note that $\pi_3(\mathfrak{C})$ is a line in $\mathbb{P}^2_{x,y,z}$.
On the~other hand, the~discriminant curve $\Delta_{\mathbb{P}^2}$ is an~irreducible smooth quartic curve in  $\mathbb{P}^2_{x,y,z}$.
Therefore, in particular, the~line  $\pi_3(\mathfrak{C})$ is not contained in $\Delta_{\mathbb{P}^2}$.
Now, let $P$ be a general point in $\mathfrak{C}$, let $Z$ be the~fiber of the~conic bundle $\pi_3$
that passes through $P$, and let $S$ be a general surface in $|H_3|$ that contains the~curve $Z$.
Then $Z$ and $S$ are both~smooth,
and it follows from Lemma~\ref{lemma:quartic-weak-del-Pezzo-surfaces} that $S$ is a del Pezzo of degree $4$,
so that $\delta_P(X)>1$ by Lemma~\ref{lemma:delta-P-del-Pezzo-degree-4-smooth}.
\end{proof}

Hence, to complete the~proof of Main Theorem, we may assume that $\mathfrak{C}$ is a point. Set $P=\mathfrak{C}$.
Let~$\mathscr{C}$ be the~fiber of the~conic bundle $\omega$ that contains~$P$.

\begin{lemma}
\label{lemma:the-proof-E1-cap-E2}
Suppose that $P\not\in E_1\cap E_2$. Then $\beta(\mathbf{F})>0$.
\end{lemma}

\begin{proof}
Apply Lemmas~\ref{lemma:S-C-P-point-not-in-E1-E2-1}, \ref{lemma:S-C-P-point-not-in-E1-E2-2},
\ref{lemma:S-C-P-point-not-in-E1-E2-3} and Corollary~\ref{corollary:AZ-twice}.
\end{proof}

Thus, to complete the~proof of Main Theorem, we may assume, in addition, that $P\in E_1\cap E_2$.
Then~the~conic $\mathscr{C}$ is smooth at $P$ by Lemma~\ref{lemma:conic-conic-singular}.
In particular, we see that $\mathscr{C}$ is reduced.

\begin{lemma}
\label{lemma:the-proof-conic-smooth}
Suppose that $\mathscr{C}$ is smooth. Then $\beta(\mathbf{F})>0$.
\end{lemma}

\begin{proof}
Apply Lemma~\ref{lemma:S-C-P-point-not-in-Delta-P1-P1} and Corollary~\ref{corollary:AZ-twice}.
\end{proof}

To complete the~proof of Main Theorem, we may assume that $\mathscr{C}$ is singular. Write $\mathscr{C}=\ell_1+\ell_2$,
where $\ell_1$ and $\ell_2$ are irreducible components of the~conic $\mathscr{C}$.
Then $P\ne\ell_1\cap\ell_2$, since $P\not\in\mathrm{Sing}(\mathscr{C})$.

Let $S_1$ and $S_2$ be general surfaces in $|H_1|$ and $|H_2|$ that passes through the~point $P$, respectively.
Then $\mathscr{C}=S_1\cap S_2$, and it follows from Corollary~\ref{corollary:S1-S2-singular} that $S_1$ or $S_2$ is smooth along the~conic $\mathscr{C}$.
Without loss of generality, we may assume that $S_1$ is smooth along $\mathscr{C}$.
We let $S=S_1$.

If $S$ is smooth, then $\delta_P(X)>1$ by Corollary~\ref{corollary:delta-P-quintic-del-Pezzo-smooth}.
Thus, we may assume that $S$ is singular.

Recall that $S$ is a quintic del Pezzo surface,
and $\ell_1$ and $\ell_2$ are lines in its anticanonical embedding.
The preimages of the~lines $\ell_1$ and $\ell_2$ on the~minimal resolution of the~surface $S$ are $(-1)$-curves,
which do not intersect $(-2)$-curves.
By Lemma~\ref{lemma:singular-quintic-del-Pezzo-surfaces} and Remark~\ref{remark:singular-quintic-del-Pezzo-surfaces},
one of the~following cases holds:
\begin{enumerate}
\item[($\mathbb{A}_1$)]  the~surface $S$ has one singular point of type $\mathbb{A}_1$,
\item[($2\mathbb{A}_1$)] the~surface $S$ has two singular points of type $\mathbb{A}_1$.
\end{enumerate}
In both cases, the~restriction morphism $\pi_3\vert_{S}\colon S\to\mathbb{P}^2_{x,y,z}$ is birational.
In ($\mathbb{A}_1$)-case, this morphism contracts three disjoint irreducible smooth rational curves $\mathbf{e}_1$, $\mathbf{e}_2$, $\mathbf{e}_3$ such that
\mbox{$E_1\vert_{S}=2\mathbf{e}_1+\mathbf{e}_2+\mathbf{e}_3$},
the curves $\mathbf{e}_1$, $\mathbf{e}_2$, $\mathbf{e}_3$ are sections of the~conic bundle $\pi_2\vert_{S}\colon S\to\mathbb{P}^1_{u,v}$,
the curve $\mathbf{e}_1$ passes through the~singular point of the~surface $S$,
but $\mathbf{e}_2$ and $\mathbf{e}_3$ are  contained in the~smooth locus of the~surface~$S$.
In~($2\mathbb{A}_1$)-case, the~morphism $\pi_3\vert_{S}$ contracts two disjoint curves $\mathbf{e}_1$ and
$\mathbf{e}_2$  such that $E_1\big\vert_{S}=2\mathbf{e}_1+2\mathbf{e}_2$,
the curves $\mathbf{e}_1$ and $\mathbf{e}_2$ are sections of the~conic bundle $\pi_2\vert_{S}$,
and each curve among $\mathbf{e}_1$ and $\mathbf{e}_2$ contains one singular point of the~surface $S$.
In both cases,  we may assume that $\ell_1\cap\mathbf{e}_1\ne\varnothing$.

Let us identify the~surface $S$ with its image in $\mathbb{P}^5$ via the~anticanonical embedding $S\hookrightarrow\mathbb{P}^5$.
Then~$\ell_1$ and $\ell_2$ and the~curves contracted by $\pi_3\vert_{S}$ are lines.
In ($\mathbb{A}_1$)-case, the~surface $S$ contains two additional lines $\ell_3$ and $\ell_4$ such that
$\ell_3+\ell_4\sim \ell_1+\ell_2$, the~intersection $\ell_3\cap\ell_4$ is the~singular point of the~surface $S$,
and the~intersection graph of the~lines $\ell_1$, $\ell_2$, $\ell_3$, $\ell_4$, $\mathbf{e}_1$, $\mathbf{e}_2$, $\mathbf{e}_3$ is shown here:
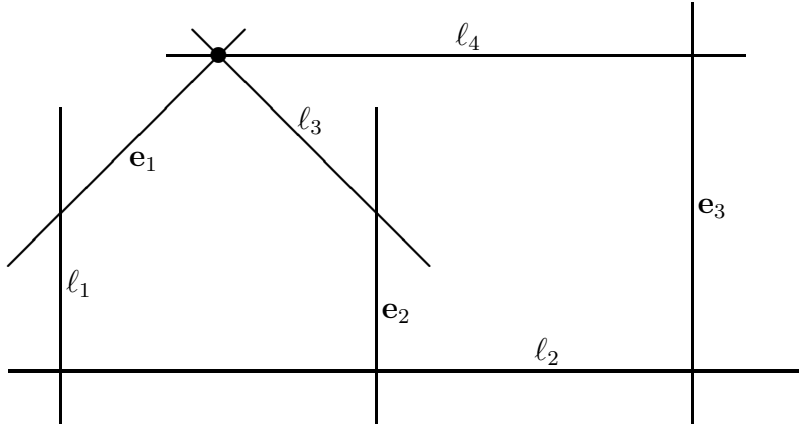
\begin{center}
\setlength{\unitlength}{0.70mm}
\begin{picture}(170,85)(0,0)
\thicklines
\put(10,10){\line(1,0){150}}
\put(20,0){\line(0,1){60}}
\put(80,0){\line(0,1){60}}
\put(140,0){\line(0,1){80}}
\put(150,70){\line(-1,0){110}}
\put(55,75){\line(-1,-1){45}}
\put(45,75){\line(1,-1){45}}
\put(50,70){\circle*{3}}
\put(33,49){\mbox{$\mathbf{e}_1$}}
\put(65,56){\mbox{$\ell_3$}}
\put(21,25){\mbox{$\ell_1$}}
\put(81,20){\mbox{$\mathbf{e}_2$}}
\put(141,40){\mbox{$\mathbf{e}_3$}}
\put(110,12){\mbox{$\ell_2$}}
\put(95,72){\mbox{$\ell_4$}}
\end{picture}
\end{center}
In this picture, we denoted by $\bullet$ the~singular point of the~surface $S$.
Moreover, on the~surface $S$, the~intersections of the~lines $\ell_1$, $\ell_2$, $\ell_3$, $\ell_4$, $\mathbf{e}_1$, $\mathbf{e}_2$, $\mathbf{e}_3$  are given in the~table below.
\begin{center}
\renewcommand\arraystretch{1.2}
\begin{tabular}{|c||c|c|c|c|c|c|c|}
\hline
$\bullet$  & $\ell_1$ & $\ell_2$ & $\ell_3$ & $\ell_4$ & $\mathbf{e}_1$ & $\mathbf{e}_2$ & $\mathbf{e}_3$\\
\hline\hline
$\ell_1$       & $-1$& $1$ & $0$& $0$& $1$& $0$ & $0$\\
\hline
$\ell_2$       & $1$& $-1$& $0$& $0$& $0$& $1$& $1$\\
\hline
$\ell_3$       & $0$& $0$& $-\frac{1}{2}$& $\frac{1}{2}$& $\frac{1}{2}$& $1$& $0$\\
\hline
$\ell_4$       & $0$& $0$& $\frac{1}{2}$& $-\frac{1}{2}$& $\frac{1}{2}$& $0$& $1$\\
\hline
$\mathbf{e}_1$ & $1$& $0$& $\frac{1}{2}$& $\frac{1}{2}$& $-\frac{1}{2}$& $0$& $0$\\
\hline
$\mathbf{e}_2$ & $0$ & $1$& $1$& $0$& $0$& $-1$& $0$\\
\hline
$\mathbf{e}_3$ & $0$& $1$& $0$& $1$& $0$& $0$& $-1$\\
\hline
\end{tabular}
\end{center}

Likewise, in ($2\mathbb{A}_1$)-case, the~surface $S$ contains one additional lines $\ell_3$ such that
$2\ell_3\sim \ell_1+\ell_2$, the~line $\ell_3$ passes through
both singular points of the~del Pezzo surface $S$, and the~intersection graph of the~lines on the~surface $S$ is shown in the~following picture:
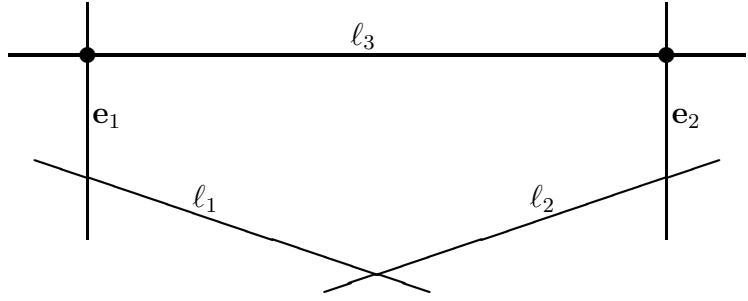
\begin{center}
\setlength{\unitlength}{0.70mm}
\begin{picture}(170,60)(0,0)
\thicklines
\put(15,50){\line(1,0){140}}
\put(80,52){\mbox{$\ell_3$}}
\put(30,50){\circle*{3}}
\put(140,50){\circle*{3}}
\put(30,60){\line(0,-1){45}}
\put(31,37){\mbox{$\mathbf{e}_1$}}
\put(140,60){\line(0,-1){45}}
\put(141,37){\mbox{$\mathbf{e}_2$}}
\put(20,30){\line(3,-1){75}}
\put(50,21){\mbox{$\ell_1$}}
\put(150,30){\line(-3,-1){75}}
\put(114,21){\mbox{$\ell_2$}}
\end{picture}
\end{center}
As above, singular points of the~surface $S$ are denote by $\bullet$.
The intersections of the~lines $\ell_1$, $\ell_2$, $\ell_3$, $\mathbf{e}_1$, $\mathbf{e}_2$ on the~surface $S$ are given in the~table below.
\begin{center}
\renewcommand\arraystretch{1.2}
\begin{tabular}{|c||c|c|c|c|c|}
\hline
$\bullet$       & $\ell_1$ & $\ell_2$ & $\ell_3$ & $\mathbf{e}_1$ & $\mathbf{e}_2$ \\
\hline\hline
$\ell_1$        & $-1$& $1$& $0$& $1$& $0$\\
\hline
$\ell_2$        & $1$& $-1$& $0$& $0$& $1$ \\
\hline
$\ell_3$        & $0$& $0$& $0$& $\frac{1}{2}$& $\frac{1}{2}$\\
\hline
$\mathbf{e}_1$  & $1$& $0$& $\frac{1}{2}$& $-\frac{1}{2}$& $0$\\
\hline
$\mathbf{e}_2$  & $0$& $1$& $\frac{1}{2}$& $0$& $-\frac{1}{2}$\\
\hline
\end{tabular}
\end{center}

\begin{remark}
\label{remarl:quintic-del-Pezzo-lines}
By \cite[Lemma~2.9]{CheltsovProkhorov},
the lines in $S$ generate the~group $\mathrm{Cl}(S)$ and the~cone of effective divisors $\mathrm{Eff}(S)$,
and every extremal ray of the~Mori cone $\overline{\mathrm{NE}}(S)$ is generated by the~class of a line.
\end{remark}

In ($\mathbb{A}_1$)-case, the~point $P$ is one of the~points $\mathbf{e}_1\cap\ell_1$, $\mathbf{e}_2\cap\ell_2$ or $\mathbf{e}_3\cap\ell_2$, because $P\in E_1\cap E_2$.
On~the other hand, if $P=\mathbf{e}_2\cap\ell_2$ or $P=\mathbf{e}_3\cap\ell_2$, it follows from Corollary~\ref{corollary:delta-P-quintic-del-Pezzo-A1} that $\delta_P(X)>1$.
In~($2\mathbb{A}_1$)-case, either $P=\mathbf{e}_1\cap\ell_1$ or $P=\mathbf{e}_2\cap\ell_2$.
Therefore, to complete the~proof of Main Theorem, we may assume that $P=\mathbf{e}_1\cap\ell_1$ in both cases.

Now, we will apply Corollary~\ref{corollary:AZ-twice} to the~surface $S$ with $C=\mathbf{e}_1$  at the~point $P$.
We have $\tau=\frac{3}{2}$.
As in the~proof of Corollary~\ref{lemma:delta-P-del-Pezzo-smooth-A1}, we see that
$$
P(u)=\left\{\aligned
&(1-u)H_1+H_2+H_3\ \text{if $0\leqslant u\leqslant 1$}, \\
&(2-u)H_2+(3-2u)H_3\ \text{if $1\leqslant u\leqslant \frac{3}{2}$}, \\
\endaligned
\right.
$$
and
$$
N(u)=\left\{\aligned
&0\ \text{if $0\leqslant u\leqslant 1$}, \\
&(u-1)E_2\ \text{if $1\leqslant u\leqslant \frac{3}{2}$}.\\
\endaligned
\right.
$$
Since $H_1\vert_{S}\sim 0$, $H_2\vert_{S}\sim\ell_1+\ell_2$, $H_3\vert_{S}\sim \ell_1+2\mathbf{e}_1$, we have
$$
P(u)\big\vert_{S}-v\mathbf{e}_1\sim_{\mathbb{R}}\left\{\aligned
&(2-v)\mathbf{e}_1+2\ell_1+\ell_2\ \text{if $0\leqslant u\leqslant 1$}, \\
&(6-4u-v)\mathbf{e}_1+(5-3u)\ell_1+(2-u)\ell_2\ \text{if $1\leqslant u\leqslant \frac{3}{2}$}. \\
\endaligned
\right.
$$
Thus, since the~intersection form of the~curves $\ell_1$ and $\ell_2$ is semi-negative definite, we get
$$
t(u)=\left\{\aligned
&2\ \text{if $0\leqslant u\leqslant 1$}, \\
&6-4u\ \text{if $1\leqslant u\leqslant \frac{3}{2}$}.\\
\endaligned
\right.
$$
Similarly, if $0\leqslant u\leqslant 1$, then
$$
P(u,v)=\left\{\aligned
&(2-v)\mathbf{e}_1+2\ell_1+\ell_2\ \text{if $0\leqslant v\leqslant 1$}, \\
&(2-v)\mathbf{e}_1+(3-v)\ell_1+\ell_2\ \text{if $1\leqslant v\leqslant 2$}, \\
\endaligned
\right.
$$
$$
N(u,v)=\left\{\aligned
&0\ \text{if $0\leqslant v\leqslant 1$}, \\
&(v-1)\ell_1\ \text{if $1\leqslant v\leqslant 2$},\\
\endaligned
\right.
$$
$$
P(u,v)\cdot\mathbf{e}_1=\left\{\aligned
&\frac{v+2}{2}\ \text{if $0\leqslant v\leqslant 1$}, \\
&\frac{4-v}{2}\ \text{if $1\leqslant v\leqslant 2$}, \\
\endaligned
\right.
$$
$$
\mathrm{vol}\big(P(u)\big\vert_{S}-v\mathbf{e}_1\big)=\left\{\aligned
&\frac{10-4v-v^2}{2}\ \text{if $0\leqslant v\leqslant 1$}, \\
&\frac{(2-v)(6-v)}{2}\ \text{if $1\leqslant v\leqslant 2$}.\\
\endaligned
\right.
$$
Likewise, if $1\leqslant u\leqslant\frac{3}{2}$, then
$$
P(u,v)=\left\{\aligned
&(6-4u-v)\mathbf{e}_1+(5-3u)\ell_1+(2-u)\ell_2\ \text{if $0\leqslant v\leqslant 3-2u$}, \\
&(6-4u-v)\mathbf{e}_1+(8-5u-v)\ell_1+(2-u)\ell_2\ \text{if $3-2u\leqslant v\leqslant 6-4u$}, \\
\endaligned
\right.
$$
$$
N(u,v)=\left\{\aligned
&0\ \text{if $0\leqslant v\leqslant 3-2u$}, \\
&(v+2u-3)\ell_1\ \text{if $3-2u\leqslant v\leqslant 6-4u$},\\
\endaligned
\right.
$$
$$
P(u,v)\cdot\mathbf{e}_1=\left\{\aligned
&\frac{4+v-2u}{2}\ \text{if $0\leqslant v\leqslant 3-2u$}, \\
&\frac{10-6u-v}{2}\ \text{if $3-2u\leqslant v\leqslant 6-4u$}, \\
\endaligned
\right.
$$
$$
\mathrm{vol}\big(P(u)\big\vert_{S}-v\mathbf{e}_1\big)=\left\{\aligned
&\frac{66+24u^2+4uv-v^2-80u-8v}{2}\ \text{if $0\leqslant v\leqslant 3-2u$}, \\
&\frac{(6-4u-v)(14-8u-v)}{2}\ \text{if $3-2u\leqslant v\leqslant 6-4u$}.\\
\endaligned
\right.
$$
Integrating, we get $S(W_{\bullet,\bullet}^S;\mathbf{e}_1)=\frac{137}{144}$ and
$S(W_{\bullet, \bullet,\bullet}^{S,\mathbf{e}_1};P)=\frac{59}{96}+F_P(W_{\bullet,\bullet,\bullet}^{S,\mathbf{e}_1})$.
To compute $F_P(W_{\bullet,\bullet,\bullet}^{S,\mathbf{e}_1})$, we let $Z=E_2\vert_{S}$.
Then $Z$ is a smooth curve of genus $3$ such that $\pi(Z)$ is a smooth quartic in $\mathbb{P}^2_{x,y,z}$.
Moreover, the~curve $Z$ is contained in the~smooth locus of the~surface $S$, and
$$
Z\sim\left\{\aligned
&4\mathbf{e}_1+\ell_3+\ell_4+2\ell_1\ \text{in ($\mathbb{A}_1$)-case}, \\
&2\ell_1+2\ell_2+2\mathbf{e}_1+2\mathbf{e}_2\ \text{in ($2\mathbb{A}_1$)-case}. \\
\endaligned
\right.
$$
In particular, we have $Z\cdot\mathbf{e}_1=1$. Since $\mathbf{e}_1\not\subset Z$, we have
$$
N_S^\prime(u)=\left\{\aligned
&0\ \text{if $0\leqslant u\leqslant 1$}, \\
&(u-1)Z\ \text{if $1\leqslant u\leqslant \frac{3}{2}$}.\\
\endaligned
\right.
$$
Note that  $P\in Z$, because $P\in E_1\cap E_2$.
Thus, since $\mathbf{e}_1\cdot Z=1$ and $\mathbf{e}_1\cdot\ell_1=1$, we have
\begin{multline*}
F_P\big(W_{\bullet,\bullet,\bullet}^{S,\mathbf{e}_1}\big)=\frac{1}{3}\int_1^{\frac{3}{2}}\int_0^{6-4u}\big(P(u,v)\cdot \mathbf{e}_1\big)(u-1)dvdu+\frac{1}{3}\int_0^{\frac{3}{2}}\int_0^{t(u)}\big(P(u,v)\cdot \mathbf{e}_1\big)\big(N(u,v)\cdot \mathbf{e}_1\big)dvdu=\\
=\frac{1}{3}\int_1^{\frac{3}{2}}\int_0^{3-2u}\frac{(4+v-2u)(u-1)}{2}dvdu+\frac{1}{3}\int_1^{\frac{3}{2}}\int_{3-2u}^{6-4u}\frac{(10-6u-v)(u-1)}{2}dvdu+\\
+\frac{1}{3}\int_0^{1}\int_1^{2}\frac{(4-v)(v-1)}{2}dvdu+\frac{1}{3}\int_1^{\frac{3}{2}}\int_{3-2u}^{6-4u}\frac{(10-6u-v)(v+2u-3)}{2}dvdu=\frac{71}{288},
\end{multline*}
so that $S(W_{\bullet, \bullet,\bullet}^{S,\mathbf{e}_1};P)=\frac{31}{36}$.
Now, applying Corollary~\ref{corollary:AZ-twice}, we get $\delta_P(X)>1$, because $S_X(S)<1$.
Therefore, we see that $\beta(\mathbf{F})>0$. By \cite{Fujita2019Crelle,Li}, this completes the~proof of Main Theorem.

\begin{remark}
\label{remark:final-step-Kento}
Instead of using Corollary~\ref{corollary:AZ-twice}, we can finish the~proof of Main Theorem as follows.
Let $F$ be a divisor over $S$ such that $P\in C_S(F)$, and let $\mathcal{C}$ be a fiber of the~conic bundle $\pi_2\vert_{S}$.
Then, arguing as in the~proof of Corollary~\ref{lemma:delta-P-del-Pezzo-smooth-A1}, we get
$$
S\big(W^S_{\bullet,\bullet};F\big)\leqslant
\Bigg(\frac{7}{288}+\frac{5}{6\delta_P(S)}\Bigg)A_S(F)+\frac{1}{6}\int_{1}^{\frac{3}{2}}\int_0^\infty\mathrm{vol}\big((2-u)\mathcal{C}+(3-2u)H_3\big\vert_{S}-vF\big)dvdu.
$$
But $\delta_P(S)=1$ by Lemmas~\ref{lemma:delta-dP-5-A1} and \ref{lemma:delta-dP-5-A1-A1}, since $P=\mathbf{e}_1\cap\ell_1$.
Thus, we have
\begin{multline}
\label{equation:final-Kento}\tag{$\heartsuit$}
S\big(W^S_{\bullet,\bullet};F\big)\leqslant\frac{247}{288}A_S(F)+\frac{1}{6}\int_{1}^{\frac{3}{2}}\int_0^\infty\mathrm{vol}\big((2-u)\mathcal{C}+(3-2u)H_3\big\vert_{S}-vF\big)dvdu=\\
=\frac{247}{288}A_S(F)+\frac{1}{6}\int_{1}^{\frac{3}{2}}(3-2u)^3\int_0^\infty\mathrm{vol}\Bigg(\frac{2-u}{3-2u}\mathcal{C}+H_3\big\vert_{S}-vF\Bigg)dvdu=\\
=\frac{247}{288}A_S(F)+\frac{1}{6}\int_{1}^{\frac{3}{2}}(3-2u)^3\int_0^\infty\mathrm{vol}\Bigg(-K_S+\frac{u-1}{3-2u}\mathcal{C}-vF\Bigg)dvdu.
\end{multline}
Set $L=-K_S+t\mathcal{C}$ for $t\in\mathbb{R}_{\geqslant 0}$. Then $L$ is ample and $L^2=5+4t$.
Define $\delta_P(S,L)$ as in Appendix~\ref{section:delta-dP}.
Then, applying \cite[Corollary~1.7.24]{ACCFKMGSSV} to the~flag $P\in \mathbf{e}_1\subset S$, we get
$$
\delta_P(S,L)\geqslant
\left\{\aligned
&1\ \text{if $0\leqslant t\leqslant \frac{-3+\sqrt{21}}{6}$}, \\
&\frac{15+12t}{6t^2+18t+13}\ \text{if $\frac{-3+\sqrt{21}}{6}\leqslant t$}.\\
\endaligned
\right.
$$
The proof of this inequality is very similar to our computations of $S(W_{\bullet,\bullet}^S;\mathbf{e}_1)$ and $S(W_{\bullet, \bullet,\bullet}^{S,\mathbf{e}_1};P)$,
so~that we omit the~details. Now, we let $t=\frac{u-1}{3-2u}$.
Then $t\geqslant \frac{-3+\sqrt{21}}{6}\iff u\geqslant\frac{3}{2}(1-\frac{1}{\sqrt{21}})$, so
\begin{multline*}
\frac{1}{6}\int_{1}^{\frac{3}{2}}(3-2u)^3\int_0^\infty\mathrm{vol}\big(-K_S+t\mathcal{C}-vF\big)dvdu=\\
=\frac{1}{6}\int_{1}^{\frac{3}{2}}(3-2u)^3(5+4t)S_{L}(F)du\leqslant
\frac{1}{6}\int_{1}^{\frac{3}{2}(1-\frac{1}{\sqrt{21}})}(3-2u)^3(5+4t)A_S(F)du+\\
+\frac{1}{6}\int_{\frac{3}{2}(1-\frac{1}{\sqrt{21}})}^{\frac{3}{2}}(3-2u)^3(5+4t)\frac{15+12t}{6t^2+18t+13}A_{S}(F)du=\frac{247}{2016}A_{S}(F).
\end{multline*}
Now, using \eqref{equation:final-Kento}, we get
$S(W^S_{\bullet,\bullet};F)\leqslant\frac{247}{288}A_S(F)+\frac{247}{2016}A_{S}(F)=\frac{247}{252}A_S(F)$.
Then $\delta_{P}(S;W^S_{\bullet,\bullet})\geqslant\frac{252}{247}$,
so that $\delta_P(X)>1$ by \eqref{equation:Hamid-Ziquan}, since $S_X(S)<1$ by \cite[Theorem~10.1]{Fujita2016}.
\end{remark}

\appendix

\section{$\delta$-invariants of del Pezzo surfaces}
\label{section:delta-dP}

In this appendix, we present three rather sporadic results about $\delta$-invariants of del Pezzo surfaces with at most du Val singularities,
which are used in the~proof of Main Theorem.

Let $S$ be a del Pezzo surface that has at most du Val singularities,
let $L$ be an ample $\mathbb{R}$-divisor on the~surface $S$, and let $P$ be a point in $S$. Set
$$
\delta_P(S,L)=\inf_{\substack{F/S\\ P\in C_S(F)}}\frac{A_{S}(F)}{S_{L}(F)},
$$
where infimum is taken over all prime divisors $F$ over $S$ such that $P\in C_S(F)$,
and
$$
S_{L}(F)=\frac{1}{L^2}\int_0^\infty \mathrm{vol}\big(L-uF\big)du.
$$

\begin{example}
\label{example:cubic-surface}
Suppose $S$ is a smooth cubic surface in $\mathbb{P}^3$, and $L=-K_S$.
Let $T$ be the hyperplane section of the cubic surface $S$ that is singular at $P$. Then it follows from \cite[Theorem~4.6]{AbbanZhuang} that
$$
\delta_P(S,L)=\left\{\aligned
&\frac{3}{2}\ \text{if $T$ is a union of three lines such that all of them contains $P$},\\
&\frac{27}{17}\ \text{if $T$ is a~union of a~line and a~conic that are tangent at $P$},\\
&\frac{5}{3}\ \text{if $T$ is an irreducible cuspidal cubic curve},\\
&\frac{18}{11}\ \text{if $T$ is a union of three lines such that only two of them contain $P$},\\
&\frac{9}{25-8\sqrt{6}}\ \text{if $T$ is a~union of a~line and a~conic that intersect transversally at $P$},\\
&\frac{12}{7}\ \text{if $T$ is an irreducible nodal cubic curve}.
\endaligned
\right.
$$
\end{example}

It would be nice to find an explicit formula for $\delta_P(S,L)$ in all possible cases.
But this problem seems to be very~difficult.
So, we will only estimate  $\delta_P(S,L)$ in three cases when $K_S^2\in\{4,5\}$.

Suppose that $4\leqslant K_S^2\leqslant 5$.  Let us identify $S$ with its image in the~anticanionical embedding.

\begin{lemma}
\label{lemma:delta-dp4}
Suppose that $S$ is smooth and $K_S^2=4$. Let $C$ be a possibly reducible conic in $S$ that passes through~$P$,
and let $L=-K_S+tC$ for $t\in\mathbb{R}_{\geqslant 0}$.
If the~conic $C$ is smooth, then
\begin{equation}
\label{equation:delta-dP4}\tag{$\clubsuit$}
\delta_P(S,L)\geqslant
\left\{\aligned
&\frac{24}{19+8t+t^2}\ \text{if $0\leqslant t\leqslant 1$}, \\
&\frac{6(1+t)}{5+6t+3t^2}\ \text{if $t\geqslant 1$}.
\endaligned
\right.
\end{equation}
Similarly, if $C$ is a reducible conic, then
\begin{equation}
\label{equation:delta-dP4-reducible-conic}\tag{$\spadesuit$}
\delta_L(S,L)\geqslant \frac{24(1+t)}{19+30t+12t^2}.
\end{equation}
\end{lemma}

\begin{proof}
The proof of this lemma is similar to the~proof of  \cite[Lemma 2.12]{ACCFKMGSSV}.
Namely, as in that proof, we will apply \cite[Theorem~1.7.1]{ACCFKMGSSV}, \cite[Corollary 1.7.12]{ACCFKMGSSV}, \cite[Corollary~1.7.25]{ACCFKMGSSV}
to get \eqref{equation:delta-dP4} and~\eqref{equation:delta-dP4-reducible-conic}.
Let us use notations introduced in \cite[\S~1]{ACCFKMGSSV} applied to $S$ polarized by the ample divisor $L$.

First, we suppose that $P$ is not contained in any line in $S$.
In particular, the conic $C$ is smooth.
Let $\sigma\colon\widetilde{S}\to S$ be the~blowup of the~point $P$, let $E$ be the exceptional curve of the blow up $\sigma$,
and let $\widetilde{C}$ be the proper transform on $\widetilde{S}$ of the conic $C$.
Then $\widetilde{S}$ is a smooth cubic surface in $\mathbb{P}^3$,
and there exists a unique line $\mathbf{l}\subset\widetilde{S}$ such that $-K_{\widetilde{S}}\sim \widetilde{C}+E+\mathbf{l}$.
Take $u\in\mathbb{R}_{\geqslant 0}$. Then
$$
\sigma^*(L)-uE\sim_{\mathbb{R}}(1+t)\widetilde{C}+(2+t-u)E+\mathbf{l},
$$
which implies that $\sigma^*(L)-uE$ is pseudoeffective $\iff$ $u\leqslant 2+t$.
Similarly, we see that
$$
\mathscr{P}(u)\sim_{\mathbb{R}}\left\{\aligned
&(1+t)\widetilde{C}+(2+t-u)E+\mathbf{l}\ \text{if $0\leqslant u\leqslant 2$}, \\
&(3+t-u)\widetilde{C}+(2+t-u)E+\mathbf{l}\ \text{if $2\leqslant u\leqslant2+t$}, \\
\endaligned
\right.
$$
$$
\mathscr{N}(u)=\left\{\aligned
&0\ \text{if $0\leqslant u\leqslant 2$}, \\
&(u-2)\widetilde{C}\ \text{if $2\leqslant u\leqslant 2+t$},\\
\endaligned
\right.
$$
$$
\mathscr{P}(u)\cdot E=\left\{\aligned
&u\ \text{if $0\leqslant u\leqslant 2$}, \\
&2\ \text{if $2\leqslant u\leqslant2+t$}, \\
\endaligned
\right.
$$
$$
\mathrm{vol}\big(\sigma^*(L)-uE\big)=\left\{\aligned
&4+4t-u^2\ \text{if $0\leqslant u\leqslant 2$}, \\
&4(2+t-u)\ \text{if $2\leqslant u\leqslant2+t$}, \\
\endaligned
\right.
$$
where we denote by $\mathscr{P}(u)$ the positive part of the Zariski decomposition of the divisor $\sigma^*(L)-uE$,
and we denote by $\mathscr{N}(u)$ its negative part. This gives
$$
S_L(E)=\frac{8+12t+3t^2}{6(1+t)}.
$$
Moreover, applying \cite[Corollary 1.7.25]{ACCFKMGSSV}, we~obtain
$$
S(W^E_{\bullet,\bullet};Q)\leqslant\frac{4+6t+3t^2}{6(1+t)}
$$
for every point $Q\in E$. Note that $A_S(E)=2$. Thus, it follows from \cite[Corollary 1.7.12]{ACCFKMGSSV} that
$$
\delta_P(S,L)\geqslant\frac{6(1+t)}{4+6t+3t^2}>\frac{24}{19+8t+t^2}.
$$

To complete the proof of the lemma, we may assume that $S$ contains a line $\ell$ such that $P\in\ell$.
Then $\ell\cdot C=0$ or $\ell\cdot C=1$. If $\ell\cdot C=0$, then $\ell$ must be an irreducible component of the conic~$C$.
Let us apply \cite[Theorem~1.7.1]{ACCFKMGSSV} and \cite[Corollary~1.7.25]{ACCFKMGSSV} to the flag $P\in\ell$ to estimate $\delta_P(S,L)$.
Take $u\in\mathbb{R}_{\geqslant 0}$. Let $P(u)$ be  the positive part of the Zariski decomposition of the divisor $L-u\ell$,
and let $N(u)$ be its negative part. We must compute $P(u)$, $N(u)$, $P(u)\cdot\ell$ and $\mathrm{vol}(L-u\ell)$,

There exists a birational morphism $\pi\colon S\to \mathbb{P}^2$ that blows up five points $O_1,\dots,O_5\in\mathbb{P}^2$
such that no three of them are collinear.
For every $i\in\{1,\ldots,5\}$, let $\mathbf{e}_i$ be the $\pi$-exceptional curve such that $\pi(\mathbf{e}_i)=O_i$.
Similarly, let $\mathbf{l}_{ij}$ be the strict transform of the~line in $\mathbb{P}^2$ that contains $O_i$ and $O_j$, where $1\leqslant i<j\leqslant 5$.
Finally, let $B$ be the~strict transform of the~conic on $\mathbb{P}^2$
that passes through the points $O_1,\dots,O_5$.
Then $\mathbf{e}_1,\ldots,\mathbf{e}_5,\mathbf{l}_{12},\ldots,\mathbf{l}_{45},B$ are all lines in $S$,
and each extremal ray of the~Mori cone $\overline{\mathrm{NE}}(S)$ is generated by a class of one of these $16$ lines.

Suppose that the conic $C$ is irreducible. Then $C\cdot\ell=1$.
In this case, without loss of generality, we may assume that $\ell=\mathbf{e}_1$ and $C\sim \mathbf{l}_{12}+\mathbf{e}_2$.
If $0\leqslant t\leqslant 1$, then
$$
P(u)=\left\{\aligned
&L-u\ell\ \text{if $0\leqslant u\leqslant 1$}, \\
&L-u\ell-(u-1)(\mathbf{l}_{12}+\mathbf{l}_{13}+\mathbf{l}_{14}+\mathbf{l}_{15})\ \text{if $1\leqslant u\leqslant 1+t$}, \\
&L-u\ell-(u-1)(\mathbf{l}_{12}+\mathbf{l}_{13}+\mathbf{l}_{14}+\mathbf{l}_{15})-(u-t-1)B \ \text{if $1+t\leqslant u\leqslant\frac{3+t}{2}$}, \\
\endaligned
\right.
$$
$$
N(u)=\left\{\aligned
&0 \ \text{if $0\leqslant u\leqslant 1$}, \\
&(u-1)(\mathbf{l}_{12}+\mathbf{l}_{13}+\mathbf{l}_{14}+\mathbf{l}_{15})\ \text{if $1\leqslant u\leqslant 1+t$}, \\
&(u-1)(\mathbf{l}_{12}+\mathbf{l}_{13}+\mathbf{l}_{14}+\mathbf{l}_{15})+(u-t-1)B \ \text{if $1+t\leqslant u\leqslant\frac{3+t}{2}$}, \\
\endaligned
\right.
$$
$$
P(u)\cdot\ell=\left\{\aligned
&1+t+u\ \text{if $0\leqslant u\leqslant 1$}, \\
&5+t-3u\ \text{if $1\leqslant u\leqslant 1+t$}, \\
&6+2t-4u\ \text{if $1+t\leqslant u\leqslant\frac{3+t}{2}$}, \\
\endaligned
\right.
$$
$$
\mathrm{vol}\big(L-u\ell\big)=\left\{\aligned
&4(1+t)-2u(1+t)-u^2 \ \text{if $0\leqslant u\leqslant 1$}, \\
&(2-u)(4+2t-3u)\ \text{if $1\leqslant u\leqslant 1+t$}, \\
&(3+t-2u)^2\ \text{if $1+t\leqslant u\leqslant\frac{3+t}{2}$}, \\
\endaligned
\right.
$$
and $L-u\ell$ is not pseudoeffective for $u>\frac{3+t}{2}$.
Similarly, if $t\geqslant 1$, then
$$
P(u)=\left\{\aligned
&L-u\ell\ \text{if $0\leqslant u\leqslant 1$}, \\
&L-u\ell-(u-1)(\mathbf{l}_{12}+\mathbf{l}_{13}+\mathbf{l}_{14}+\mathbf{l}_{15})\ \text{if $1\leqslant u\leqslant 2$}, \\
\endaligned
\right.
$$
$$
N(u)=\left\{\aligned
&0 \ \text{if $0\leqslant u\leqslant 1$}, \\
&(u-1)(\mathbf{l}_{12}+\mathbf{l}_{13}+\mathbf{l}_{14}+\mathbf{l}_{15}) \ \text{if $1\leqslant u\leqslant 2$}, \\
\endaligned
\right.
$$
$$
P(u)\cdot\ell=\left\{\aligned
&1+t+u\ \text{if $0\leqslant u\leqslant 1$}, \\
&5+t-3u\ \text{if $1\leqslant u\leqslant 2$}, \\
\endaligned
\right.
$$
$$
\mathrm{vol}\big(L-u\ell\big)=\left\{\aligned
&4(1+t)-2u(1+t)-u^2 \ \text{if $0\leqslant u\leqslant 1$}, \\
&(2-u)(4+2t-3u)\ \text{if $1\leqslant u\leqslant 2$}, \\
\endaligned
\right.
$$
and $L-u\ell$ is not pseudoeffective for $u>2$.
Then
$$
S_L\big(\ell\big)=
\left\{\aligned
&\frac{17+4t-t^2}{24}\ \text{if $0\leqslant t\leqslant 1$}, \\
&\frac{2+3t}{3(1+t)}\ \text{if $t\geqslant 1$}. \\
\endaligned
\right.
$$
Observe that $P\not\in\mathbf{l}_{ij}$ for every $1\leqslant i<j\leqslant 5$.
Thus, if $t\leqslant 1$, then \cite[Corollary~1.7.25]{ACCFKMGSSV} gives
$$
S(W^{\ell}_{\bullet,\bullet};P)=
\left\{\aligned
&\frac{19+8t+t^2}{24}\ \text{if $P\in B$}, \\
&\frac{9+15t+3t^2+t^3}{12(1+t)}\ \text{if $P\not\in B$}. \\
\endaligned
\right.
$$
Similarly, if $t\geqslant 1$, then \cite[Corollary~1.7.25]{ACCFKMGSSV} gives
$$
S\big(W^{\ell}_{\bullet,\bullet};P\big)=\frac{5+6t+3t^2}{6(1+t)}.
$$
Now, using \cite[Theorem~1.7.1]{ACCFKMGSSV}, we get \eqref{equation:delta-dP4}.

To complete the proof of the lemma, we may assume that the conic $C$ is reducible.
In this case, we let $\ell$ be an irreducible component of the conic $C$ that contains $P$.
Without loss of generality, we may assume that $\ell=\mathbf{e}_1$ and $C=\mathbf{e}_1+B$.
Then
$$
P(u)=\left\{\aligned
&L-u\ell\ \text{if $0\leqslant u\leqslant 1$}, \\
&L-u\ell-(u-1)B\ \text{if $1\leqslant u\leqslant 1+t$}, \\
&L-u\ell-(u-t-1)(\mathbf{l}_{12}+\mathbf{l}_{13}+\mathbf{l}_{14}+\mathbf{l}_{15})-(u-1)B \ \text{if $1+t\leqslant u\leqslant\frac{3+2t}{2}$}, \\
\endaligned
\right.
$$
$$
N(u)=\left\{\aligned
&0 \ \text{if $0\leqslant u\leqslant 1$}, \\
&(u-1)B\ \text{if $1\leqslant u\leqslant 1+t$}, \\
&(u-t-1)(\mathbf{l}_{12}+\mathbf{l}_{13}+\mathbf{l}_{14}+\mathbf{l}_{15})+(u-1)B \ \text{if $1+t\leqslant u\leqslant\frac{3+2t}{2}$}, \\
\endaligned
\right.
$$
$$
P(u)\cdot\ell=\left\{\aligned
&1+u\ \text{if $0\leqslant u\leqslant 1$}, \\
&2\ \text{if $1\leqslant u\leqslant 1+t$}, \\
&6+4t-4u\ \text{if $1+t\leqslant u\leqslant\frac{3+2t}{2}$}, \\
\endaligned
\right.
$$
$$
\mathrm{vol}\big(L-u\ell\big)=\left\{\aligned
&4(1+t)-2u-u^2\ \text{if $0\leqslant u\leqslant 1$}, \\
&5+4t-4u\ \text{if $1\leqslant u\leqslant 1+t$}, \\
&(3+2t-2u)^2\ \text{if $1+t\leqslant u\leqslant\frac{3+2t}{2}$}, \\
\endaligned
\right.
$$
and the divisor $L-u\ell$ is not pseudoeffective for $u>\frac{3+2t}{2}$. This gives
$$
S_L\big(\ell\big)=\frac{17+30t+12t^2}{24(1+t)}.
$$
Moreover, using \cite[Corollary~1.7.25]{ACCFKMGSSV}, we compute
$$
S\big(W^{\ell}_{\bullet,\bullet};P\big)=
\left\{\aligned
&\frac{19+30t+12t^2}{24(1+t)}\ \text{if $P\in B$}, \\
&\frac{19+24t}{24(1+t)}\ \text{if $P\in\mathbf{l}_{12}\cup\mathbf{l}_{13}\cup\mathbf{l}_{14}\cup\mathbf{l}_{15}$}, \\
&\frac{3+4t}{4(1+t)}\ \text{otherwise}. \\
\endaligned
\right.
$$
Now, using \cite[Theorem~1.7.1]{ACCFKMGSSV}, we get \eqref{equation:delta-dP4-reducible-conic} as claimed.
\end{proof}

In the remaining part of this appendix, we suppose that $K_S^2=5$, $L=-K_S$, and $S$ has isolated ordinary double points,
i.e. singular points of type $\mathbb{A}_1$. As usual, we set $\delta_P(S)=\delta_P(S,-K_S)$ and
$$
\delta(S)=\inf_{P\in S}\delta_P(S).
$$
Let $\eta\colon\widetilde{S}\to S$ be the minimal resolution of the quintic del Pezzo surface $S$.
Since $-K_{\widetilde{S}}\sim\eta^*(-K_S)$, we can estimate the number $\delta_P(S)$ as follows.
Let $O$ be a point in the surface $\widetilde{S}$ such that~$\eta(O)=P$,
and let $C$ be a smooth irreducible rational curve in $\widetilde{S}$ such that
\begin{itemize}
\item if $P\in\mathrm{Sing}(S)$, then $C$ is the~$\eta$-exceptional curve such that $\eta(C)=P$,
\item if $P\not\in\mathrm{Sing}(S)$, then $C$ is appropriately chosen curve that contains $O$.
\end{itemize}
As usual, we set
$$
\tau=\mathrm{sup}\Big\{u\in\mathbb{Q}_{\geqslant 0}\ \big\vert\ \text{the divisor  $-K_{\widetilde{S}}-uC$ is pseudo-effective}\Big\}.
$$
For~$u\in[0,\tau]$, let $P(u)$ be the~positive part of the~Zariski decomposition of the~divisor $-K_{\widetilde{S}}-uC$,
and let $N(u)$ be its negative part. Let
$$
S_{S}(C)=\frac{1}{K_S^2}\int_{0}^{\infty}\mathrm{vol}\big(-K_{\widetilde{S}}-uC\big)du=\frac{1}{K_S^2}\int_{0}^{\tau}P(u)^2du
$$
and let
$$
S\big(W^{C}_{\bullet,\bullet},O\big)=
\frac{2}{K_S^2}\int_0^\tau\big(P(u)\cdot C\big)\mathrm{ord}_O\big(N(u)\big\vert_{C}\big)du
+\frac{1}{K_S^2}\int_0^\tau(P(u)\cdot C)^2du.
$$
If $P\not\in\mathrm{Sing}(S)$, then \cite[Theorem~1.7.1]{ACCFKMGSSV} and \cite[Corollary~1.7.25]{ACCFKMGSSV} give
\begin{equation}
\label{equation:dP5-delta-estimate-1}\tag{$\blacklozenge$}
\frac{1}{S_S(C)}\geqslant\delta_P(S)\geqslant\min\left\{\frac{1}{S_S(C)},\frac{1}{S\big(W^{C}_{\bullet,\bullet},O\big)}\right\}.
\end{equation}
Similarly, if $P\in\mathrm{Sing}(S)$, then \cite[Corollary 1.7.12]{ACCFKMGSSV} and \cite[Corollary~1.7.25]{ACCFKMGSSV} give
\begin{equation}
\label{equation:dP5-delta-estimate-2}\tag{$\lozenge$}
\frac{1}{S_S(C)}\geqslant\delta_P(S)\geqslant\min\left\{\frac{1}{S_S(C)},\inf_{O\in C}\frac{1}{S\big(W^{C}_{\bullet,\bullet},O\big)}\right\}.
\end{equation}

\begin{lemma}
\label{lemma:delta-dP-5-A1}
Suppose $S$ has one singular point.
Then $\delta(S)=\frac{15}{17}$, and the following assertions hold:
\begin{itemize}
\item If $P$ is not contained in any line in $S$ that contains the singular point of $S$, then $\delta_P(S)\geqslant\frac{15}{13}$.
\item If $P$ is not the~singular point of the surface $S$,
but $P$ is contained in a line in $S$ that passes through the singular point of the surface $S$, then $\delta_P(S)=1$.
\item If $P$ is the~singular point of the surface $S$, then $\delta_P(S)=\frac{15}{17}$.
\end{itemize}
\end{lemma}

\begin{proof}
We let $P_0$ be the singular point of the surface $S$,
and let $\ell_0$ be the $\pi$-exceptional curve.
Then~it follows from \cite{Coray1988} that there exists
a birational morphism $\pi\colon\widetilde{S}\to\mathbb{P}^2$ such that $\pi(\ell_0)$ is a line,
the~map $\pi$ blows up three points $Q_1$, $Q_2$, $Q_3$ contained in  $\pi(\ell_0)$ and another point $Q_0\in\mathbb{P}^2\setminus\pi(\ell_0)$.

For $i\in\{0,1,2,3\}$, let $\mathbf{e}_i$ be the~$\pi$-exceptional curve such that $\pi(\mathbf{e}_i)=Q_i$.
For every $i\in\{1,2,3\}$, let $\ell_i$ be the strict transform of the line in $\mathbb{P}^2$ that passes through $Q_0$ and $Q_i$.
Then $\ell_0$, $\ell_1$, $\ell_2$, $\ell_3$, $\mathbf{e}_0$, $\mathbf{e}_1$, $\mathbf{e}_2$, $\mathbf{e}_3$
are the only irreducible curves in the surface $\widetilde{S}$ that have negative self-intersections.
Moreover, the intersections of these curves are given in the~following table:
\begin{center}
\renewcommand\arraystretch{1.2}
\begin{tabular}{|c||c|c|c|c|c|c|c|c|}
 \hline
 & $\ell_0$ & $\ell_1$ & $\ell_2$ & $\ell_3$ & $\mathbf{e}_0$ & $\mathbf{e}_1$ & $\mathbf{e}_2$ & $\mathbf{e}_3$ \\
 \hline
 \hline
$\ell_0$ & $-2$ & $0$ & $0$ & $0$  & $0$  & $1$ & $1$ & $1$ \\
 \hline
$\ell_1$ & $0$ & $-1$ & $0$ & $0$ & $1$ & $1$ & $0$ & $0$ \\
 \hline
$\ell_2$ & $0$ & $0$ & $-1$ & $0$ & $1$ & $0$ & $1$ & $0$ \\
 \hline
$\ell_3$ & $0$ & $0$ & $0$ & $-1$ & $1$ & $0$ & $0$ & $1$ \\
 \hline
$\mathbf{e}_0$ & $0$ & $1$ & $1$ & $1$ & $-1$ & $0$ & $0$ & $0$ \\
 \hline
$\mathbf{e}_1$ & $1$ & $1$ & $0$ & $0$ & $0$ & $-1$ & $0$ & $0$ \\
 \hline
$\mathbf{e}_2$ & $1$ & $0$ & $1$ & $0$ & $0$ & $0$ & $-1$ & $0$ \\
 \hline
$\mathbf{e}_3$ & $1$ & $0$ & $0$ & $1$ & $0$ & $0$ & $0$ & $-1$\\
 \hline
\end{tabular}
\end{center}

Note that $\eta(\ell_1)$, $\eta(\ell_2)$, $\eta(\ell_3)$, $\eta(\mathbf{e}_0)$, $\eta(\mathbf{e}_1)$, $\eta(\mathbf{e}_2)$, $\eta(\mathbf{e}_3)$
are all lines contained in the surface $S$.
Among them, only the lines $\eta(\mathbf{e}_1)$, $\eta(\mathbf{e}_2)$, $\eta(\mathbf{e}_3)$  pass through the singular point $P_0$.

For $(a_0,a_1,a_2,a_3,b_0,b_1,b_2,b_3)\in\mathbb{R}^8$, we write
$$
[a_0,a_1,a_2,a_3,b_0,b_1,b_2,b_3] := \sum_{i=0}^3 a_i \ell_i + \sum_{i=0}^3 b_i \mathbf{e}_i \in \mathrm{Pic} (\widetilde{S}) \otimes \mathbb{R}.
$$

If $P=P_0$, then $C=\ell_0$, which implies that $\tau=2$ and
$$
P(u)=\left\{\aligned
&[-u, 1, 1, 1, 2, 0, 0, 0]\ \text{if $0\leqslant u\leqslant 1$}, \\
&[-u, 1, 1, 1, 2, 1-u, 1-u, 1-u]\ \text{if $1\leqslant u\leqslant 2$}, \\
\endaligned
\right.
$$
$$
N(u)=\left\{\aligned
&0\ \text{if $0\leqslant u\leqslant 1$}, \\
&(u-1)(\mathbf{e}_1+\mathbf{e}_2+\mathbf{e}_3)\ \text{if $1\leqslant u\leqslant 2$}, \\
\endaligned
\right.
$$
$$
P(u)\cdot C=\left\{\aligned
&2\ \text{if $0\leqslant u\leqslant 1$}, \\
&3 -u\ \text{if $1\leqslant u\leqslant 2$}, \\
\endaligned
\right.
\quad
P(u)^2=\left\{\aligned
&5 - 2 u^2\ \text{if $0\leqslant u\leqslant 1$}, \\
&(4-u)(2-u)\ \text{if $1\leqslant u\leqslant 2$}, \\
\endaligned
\right.
$$
which implies that $S_S(C)=\frac{17}{15}$ and $S(W^{C}_{\bullet, \bullet};O)=1$.
Therefore, using \eqref{equation:dP5-delta-estimate-2}, we obtain $\delta_{P_0}(S)=\frac{15}{17}$.

To proceed, we may assume that $P\ne P_0$. If $O\in\mathbf{e}_0$, we let $C=\mathbf{e}_0$.
Then $\tau=2$, and
$$
P(u)=\left\{\aligned
& [0, 1, 1, 1, 2-u, 0, 0, 0]\ \text{if $0\leqslant u\leqslant 1$}, \\
&[0, 2-u, 2-u, 2-u, 2-u, 0, 0, 0]\ \text{if $1\leqslant u\leqslant 2$}, \\
\endaligned
\right.
$$
$$
N(u)=\left\{\aligned
&0\ \text{if $0\leqslant u\leqslant 1$}, \\
&(u-1) (\ell_1 + \ell_2 + \ell_3)\ \text{if $1\leqslant u\leqslant 2$}, \\
\endaligned
\right.
$$
$$
P(u)\cdot C=\left\{\aligned
&1+u\ \text{if $0\leqslant u\leqslant 1$}, \\
&4-2u\ \text{if $1\leqslant u\leqslant 2$}, \\
\endaligned
\right.
\quad
P(u)^2=\left\{\aligned
&5-2u-u^2\ \text{if $0\leqslant u\leqslant 1$}, \\
&2(2-u)^2\ \text{if $1\leqslant u\leqslant 2$}, \\
\endaligned
\right.
$$
which implies that $S_S(C)=\frac{13}{15}$ and $S(W^{C}_{\bullet, \bullet};O)\leqslant\frac{13}{15}$,
so that $\delta_P(S)=\frac{15}{13}$ by \eqref{equation:dP5-delta-estimate-1}.

If $O\in \ell_1$, we let $C=\ell_1$.
In this case, we have $\tau=2$, and
$$
P(u)=\left\{\aligned
&[0, 1-u, 1, 1, 2, 0, 0, 0]\ \text{if $0\leqslant u\leqslant 1$}, \\
&[1-u, 1-u, 1, 1, 3-u, 2-2u, 0, 0]\ \text{if $1\leqslant u\leqslant 2$}, \\
\endaligned
\right.
$$
$$
N(u)=\left\{\aligned
&0\ \text{if $0\leqslant u\leqslant 1$}, \\
&(u-1)(\ell_0+\mathbf{e}_0+2\mathbf{e}_1)\ \text{if $1\leqslant u\leqslant 2$}, \\
\endaligned
\right.
$$
$$
P(u)\cdot C=\left\{\aligned
&1+u\ \text{if $0\leqslant u\leqslant 1$}, \\
&4-2u\ \text{if $1\leqslant u\leqslant 2$}, \\
\endaligned
\right.
\quad
P(u)^2=\left\{\aligned
&5-2u-u^2\ \text{if $0\leqslant u\leqslant 1$}, \\
&2(2-u)^2\ \text{if $1\leqslant u\leqslant 2$}, \\
\endaligned
\right.
$$
so that $S_S(C)=\frac{13}{15}$.
If  $O\in\ell_1\setminus (\mathbf{e}_0\cup\mathbf{e}_1)$, then $S(W^{C}_{\bullet, \bullet};O)=\frac{11}{15}$.
If  $O=\ell_1\cap\mathbf{e}_1$, then $S(W^{C}_{\bullet, \bullet};O)=1$.
Thus, using \eqref{equation:dP5-delta-estimate-1}, we see that $\delta_P(S)=\frac{15}{13}$ if $O\in\ell_1\setminus\mathbf{e}_1$,
and $\delta_P(S)\geqslant 1$ if $O=\ell_1\cap\mathbf{e}_1$.

Similarly, $\delta_P(S)=\frac{15}{13}$ if $O\in\ell_2\setminus\mathbf{e}_2$ or $O\in\ell_3\setminus\mathbf{e}_3$,
and $\delta_P(S)\geqslant 1$ if $O=\ell_2\cap\mathbf{e}_2$ or $O=\ell_3\cap\mathbf{e}_3$.

If $O\in\mathbf{e}_1$, we let $C=\mathbf{e}_1$.
In this case, we have $\tau=2$, and
$$
P(u)=\left\{\aligned
&\Big[-\frac{u}{2}, 1, 1, 1, 2, -u, 0, 0\Big]\ \text{if $0\leqslant u\leqslant 1$}, \\
&\Big[-\frac{u}{2}, 2-u, 1, 1, 2, -u, 0, 0\Big]\ \text{if $1\leqslant u\leqslant 2$}, \\
\endaligned
\right.
$$
$$
N(u)=\left\{\aligned
&\frac{u}{2} \ell_0\ \text{if $0\leqslant u\leqslant 1$}, \\
&\frac{u}{2} \ell_0 + (u-1) \ell_1\ \text{if $1\leqslant u\leqslant 2$}, \\
\endaligned
\right.
$$
$$
P(u)\cdot C=\left\{\aligned
&\frac{2+u}{2}\ \text{if $0\leqslant u\leqslant 1$}, \\
&\frac{4 - u}{2}\ \text{if $1\leqslant u\leqslant 2$}, \\
\endaligned
\right.
\quad
P(u)^2=\left\{\aligned
&5-2u-\frac{u^2}{2}\ \text{if $0\leqslant u\leqslant 1$}, \\
&\frac{(6-u)(2-u)}{2} \ \text{if $1\leqslant u\leqslant 2$}, \\
\endaligned
\right.
$$
which implies that $S_S(C)=1$ and $S(W^{C}_{\bullet, \bullet};O)\leqslant\frac{13}{15}$ if $O\in\mathbf{e}_1\setminus \ell_0$,
so that $\delta_P(S)=1$ by \eqref{equation:dP5-delta-estimate-1}.

Likewise, we see that $\delta_P(S)=1$ in the case when $O\in\mathbf{e}_2$ or $O\in\mathbf{e}_3$.
Thus, to complete the~proof, we may assume that $P$ is not contained in any line in $S$.

Now, we let $C$ be the unique curve in the pencil $|\ell_1+\mathbf{e}_1|$ that contains $P$.
By our assumption, the curve $C$ is smooth and irreducible.
Then $\tau=2$, and
$$
P(u)=\left\{\aligned
&\Big[-\frac{u}{2}, 1-u, 1, 1, 2, -u, 0, 0\Big]\ \text{if $0\leqslant u\leqslant 1$}, \\
&\Big[-\frac{u}{2}, 1-u, 1, 1, 3-u, -u, 0, 0\Big]\ \text{if $1\leqslant u\leqslant 2$}, \\
\endaligned
\right.
$$
$$
N(u)=\left\{\aligned
&\frac{u}{2} \ell_0\ \text{if $0\leqslant u\leqslant 1$}, \\
&\frac{1}{2} u \ell_0 + (u-1)\mathbf{e}_0\ \text{if $1\leqslant u\leqslant 2$}, \\
\endaligned
\right.
$$
$$
P(u)\cdot C=\left\{\aligned
&\frac{4-u}{2}\ \text{if $0\leqslant u\leqslant 1$}, \\
&\frac{3(2-u)}{2}\ \text{if $1\leqslant u\leqslant 2$}, \\
\endaligned
\right.
\quad
P(u)^2=\left\{\aligned
&5-4u+\frac{u^2}{2}\ \text{if $0\leqslant u\leqslant 1$}, \\
&\frac{3(2-u)^2}{2}\ \text{if $1\leqslant u\leqslant 2$}. \\
\endaligned
\right.
$$
Then $S_S(C)=\frac{11}{15}$ and $S(W^{C}_{\bullet, \bullet};O)=\frac{23}{30}$.
Thus, it follows from \eqref{equation:dP5-delta-estimate-1} that $\delta_P (S)\geqslant\frac{30}{23}>\frac{15}{13}$.
\end{proof}

Finally, let us estimate $\delta_P(S)$ in the case when the~del Pezzo surface $S$ has two singular points.
In this case, the surface $S$ contains a line that passes through both its singular points \cite{Coray1988}.

\begin{lemma}
\label{lemma:delta-dP-5-A1-A1}
Suppose $S$ has two singular points.
Let $\ell$ be the~line in $S$ that passes through both singular points of the surface $S$.
Then $\delta(S)=\frac{15}{19}$. Moreover, the following assertions hold:
\begin{itemize}
\item If $P$ is not contained in any line in $S$ that contains a singular point of $S$, then $\delta_P(S)\geqslant\frac{15}{13}$.
\item If $P$ is not contained in the line $\ell$, but $P$ is contained in a line in $S$ that passes through a~singular point of the surface $S$, then $\delta_P(S)=1$.
\item If $P\in\ell$, then $\delta_P(S)=\frac{15}{19}$.
\end{itemize}
\end{lemma}

\begin{proof}
Let $\mathbf{e}_1$ and $\mathbf{e}_2$ be $\eta$-exceptional curves.
Then $\widetilde{S}$
contains $(-1)$-curves $\ell_1$, $\ell_2$, $\ell_3$, $\ell_4$, $\ell_5$ such that the intersections of the curves
$\ell_1$, $\ell_2$, $\ell_3$, $\ell_4$, $\ell_5$,  $\mathbf{e}_1$, $\mathbf{e}_2$ on  $\widetilde{S}$ are given in the following table.
\begin{center}
\renewcommand\arraystretch{1.2}
\begin{tabular}{|c||c|c|c|c|c|c|c|}
 \hline
 & $\ell_1$ & $\ell_2$ & $\ell_3$ & $\ell_4$ & $\ell_5$ & $\mathbf{e}_1$ & $\mathbf{e}_2$  \\
 \hline
 \hline
$\ell_1$ & $-1$ & $0$ & $0$ & $0$ & $0$ & $1$ & $1$ \\
 \hline
$\ell_2$ & $0$& $-1$ & $1$ & $0$& $0$& $1$ & $0$\\
 \hline
$\ell_3$ & $0$ & $1$ & $-1$ & $1$ & $0$ & $0$ & $0$\\
 \hline
$\ell_4$ & $0$ & $0$ & $1$ & $-1$ & $1$ & $0$ & $0$ \\
 \hline
$\ell_5$ & $0$ & $0$ & $0$ & $1$ & $-1$ & $0$ & $1$ \\
 \hline
$\mathbf{e}_1$ & $1$ & $1$ & $0$ & $0$ & $0$ & $-2$ & $0$\\
 \hline
$\mathbf{e}_2$ & $1$ & $0$ & $0$ & $0$ & $1$ & $0$ & $-2$ \\
 \hline
\end{tabular}
\end{center}

The curves $\eta(\ell_1)$, $\eta(\ell_2)$, $\eta(\ell_3)$, $\eta(\ell_4)$, $\eta(\ell_5)$ are the only lines in $S$.
Moreover, we have $\ell=\eta(\ell_1)$,
and $\eta(\ell_1)$, $\eta(\ell_2)$, $\eta(\ell_5)$ are the only lines in $S$
that contain a singular point of the surface $S$.

As in the proof of Lemma~\ref{lemma:delta-dP-5-A1}, for $(a_1,a_2,a_3,a_4,a_5,b_1, b_2)\in \mathbb{R}^7$, we write
$$
[a_1,a_2,a_3,a_4,a_5,b_1, b_2] := \sum_{i=1}^5 a_i \ell_i + \sum_{i=1}^2 b_i \mathbf{e}_i \in \mathrm{Pic} (\widetilde{S}) \otimes \mathbb{R}.
$$

If $O\in\ell_1\setminus(\mathbf{e}_1\cup\mathbf{e}_2)$, we let $C=\ell_1$.
In this case, we have $\tau=3$, and
$$
P(u)=\left\{\aligned
&\Big[1-u, 1, 1, 1, 1, \frac{2-u}{2}, \frac{2-u}{2}\Big] \ \text{if $0\leqslant u\leqslant 2$}, \\
&[1-u, 3-u, 3-u, 0, 0, 0]\ \text{if $2\leqslant u\leqslant 3$}, \\
\endaligned
\right.
$$
$$
N(u)=\left\{\aligned
&\frac{u}{2} (\mathbf{e}_1 + \mathbf{e}_2)\ \text{if $0\leqslant u\leqslant 2$}, \\
&(u-2) (\ell_2 + \ell_5) + (u-1)(\mathbf{e}_1 + \mathbf{e}_2) \ \text{if $2\leqslant u\leqslant 3$}, \\
\endaligned
\right.
$$
$$
P(u)\cdot C=\left\{\aligned
&1 \ \text{if $0\leqslant u\leqslant 2$}, \\
&3-u \ \text{if $2\leqslant u\leqslant 3$}, \\
\endaligned
\right.
\quad
P(u)^2=\left\{\aligned
&5 - 2u \ \text{if $0\leqslant u\leqslant 2$}, \\
&(3-u)^2 \ \text{if $2\leqslant u\leqslant 3$}, \\
\endaligned
\right.
$$
which implies that $S_S(C)=\frac{19}{15}$ and $S(W^{C}_{\bullet, \bullet};O)\leqslant\frac{17}{15}$,
so that $\delta_P(S)=\frac{15}{19}$ by \eqref{equation:dP5-delta-estimate-1}.

If $O\in\mathbf{e}_1$, then $C=\mathbf{e}_1$.
In this case, we have $\tau=2$, and
$$
P(u)=\left\{\aligned
&[1, 1, 1, 1, 1, 1-u, 1]\ \text{if $0\leqslant u\leqslant 1$}, \\
&[3-2u, 2-u, 1, 1, 1, 1-u, 2-u]\ \text{if $1\leqslant u\leqslant 2$}, \\
\endaligned
\right.
$$
$$
N(u)=\left\{\aligned
&0\ \text{if $0\leqslant u\leqslant 1$}, \\
&2(u-1)\ell_1 + (u-1) \ell_2 + (u-1) \mathbf{e}_2\ \text{if $1\leqslant u\leqslant 2$}, \\
\endaligned
\right.
$$
$$
P(u)\cdot C=\left\{\aligned
&2u\ \text{if $0\leqslant u\leqslant 1$}, \\
&3-u\ \text{if $1\leqslant u\leqslant 2$}, \\
\endaligned
\right.
\quad
P(u)^2=\left\{\aligned
& 5-2u^2\ \text{if $0\leqslant u\leqslant 1$}, \\
&(2-u)(4-u) \ \text{if $1\leqslant u\leqslant 2$}, \\
\endaligned
\right.
$$
which implies that $S_S(C)=\frac{17}{15}$ and $S(W^{C}_{\bullet, \bullet};O)\leqslant\frac{19}{15}$,
so that $\delta_P(S)\geqslant\frac{19}{15}$ by \eqref{equation:dP5-delta-estimate-2}.

On the other hand, we already know that $S_S(\ell)=\frac{19}{15}$, which implies that $\delta_P(S)=\frac{19}{15}$ if $P=\eta(\mathbf{e}_1)$.
Similarly, we see that $\delta_P(S)=\frac{19}{15}$ if $P=\eta(\mathbf{e}_2)$.
Hence, we may assume that $O\not\in\mathbf{e}_1\cup\mathbf{e}_2\cup\ell_1$.

If $O\in \ell_2$, we let $C=\ell_2$.
In this case, we have $\tau=2$, and
$$
P(u)=\left\{\aligned
&\Big[1, 1-u, 1, 1, 1, \frac{2-u}{2}, 1\Big]\ \text{if $0\leqslant u\leqslant 1$}, \\
&\Big[1, 1-u, 2-u, 1, 1, \frac{2-u}{2}, 1\Big]\ \text{if $1\leqslant u\leqslant 2$}, \\
\endaligned
\right.
$$
$$
N(u)=\left\{\aligned
&\frac{u}{2} \mathbf{e}_1\ \text{if $0\leqslant u\leqslant 1$}, \\
&\frac{u}{2} \mathbf{e}_1 + (u-1) \ell_3\ \text{if $1\leqslant u\leqslant 2$}, \\
\endaligned
\right.
$$
$$
P(u)\cdot C=\left\{\aligned
&\frac{2+u}{2}\ \text{if $0\leqslant u\leqslant 1$}, \\
&\frac{4-u}{2}\ \text{if $1\leqslant u\leqslant 2$}, \\
\endaligned
\right.
\quad
P(u)^2=\left\{\aligned
& 5-2u - \frac{u^2}{2}\ \text{if $0\leqslant u\leqslant 1$}, \\
&\frac{(6-u)(2-u)}{2}\ \text{if $1\leqslant u\leqslant 2$}, \\
\endaligned
\right.
$$
which implies that $S_S(C)=1$ and $S(W^{C}_{\bullet, \bullet};O)\leqslant\frac{13}{15}$,
so that $\delta_P(S)=1$ by \eqref{equation:dP5-delta-estimate-1}.

Similarly, we see that $\delta_P(S)=1$ if $O\in \ell_5$.
Hence, if $P$ is contained in a line in $S$ that passes through a~singular point of the surface $S$, then $\delta_P(S)=1$.
Thus, we may assume that $O\not\in\ell_2\cup\ell_2$.

If $P\in \ell_3$, we let $C=\ell_3$.
In this case, we have $\tau=2$, and
$$
P(u)=\left\{\aligned
&[1, 1, 1-u, 1, 1, 1, 1]\ \text{if $0\leqslant u\leqslant 1$}, \\
& [1, 3-2u, 1-u, 2-u, 1, 2-u, 1]\ \text{if $1\leqslant u\leqslant 2$}, \\
\endaligned
\right.
$$
$$
N(u)=\left\{\aligned
&0\ \text{if $0\leqslant u\leqslant 1$}, \\
& (u-1) (\ell_4 + 2 \ell_2 + \mathbf{e}_1)\ \text{if $1\leqslant u\leqslant 2$}, \\
\endaligned
\right.
$$
$$
P(u)\cdot C=\left\{\aligned
&1+u\ \text{if $0\leqslant u\leqslant 1$}, \\
&4-2u\ \text{if $1\leqslant u\leqslant 2$}, \\
\endaligned
\right.
\quad
P(u)^2=\left\{\aligned
&5-2u - u^2\ \text{if $0\leqslant u\leqslant 1$}, \\
&2(2-u)^2\ \text{if $1\leqslant u\leqslant 2$}, \\
\endaligned
\right.
$$
which implies that $S_S(C)=\frac{13}{15}$ and $S(W^{C}_{\bullet, \bullet};O)\leqslant\frac{13}{15}$,
so that $\delta_P(S)=\frac{15}{13}$ by \eqref{equation:dP5-delta-estimate-1}.

Similarly, we see that $\delta_P(S)=\frac{15}{13}$ if $O\in\ell_4$.
Therefore, we may also assume that $O\not\in\ell_3\cup\ell_4$.

Let $C$ be the curve in the pencil $|\ell_2 + \ell_3|$ that contains $O$.
Then $C$ is smooth and irreducible, since $O$ is not contained in the curves $\ell_1$, $\ell_2$, $\ell_3$, $\ell_4$, $\ell_5$,  $\mathbf{e}_1$, $\mathbf{e}_2$ by assumption.
Then $\tau=2$, and
$$
P(u)=\left\{\aligned
&\Big[1, 1-u, 1-u, 1, 1, \frac{2-u}{2}, 1\Big]\ \text{if $0\leqslant u\leqslant 1$}, \\
&\Big[1, 1-u, 1-u, 2-u, 1, \frac{2-u}{2}, 1\Big]\ \text{if $1\leqslant u\leqslant 2$}, \\
\endaligned
\right.
$$
$$
N(u)=\left\{\aligned
&\frac{u}{2} \mathbf{e}_1\ \text{if $0\leqslant u\leqslant 1$}, \\
&\frac{u}{2} \mathbf{e}_1 + (u-1) \ell_4\ \text{if $1\leqslant u\leqslant 2$}, \\
\endaligned
\right.
$$
$$
P(u)\cdot C=\left\{\aligned
&\frac{4-u}{2}\ \text{if $0\leqslant u\leqslant 1$}, \\
&\frac{3(2-u)}{2} \ \text{if $1\leqslant u\leqslant 2$}, \\
\endaligned
\right.
\quad
P(u)^2=\left\{\aligned
&5-4u+\frac{u^2}{2}\ \text{if $0\leqslant u\leqslant 1$}, \\
&\frac{3(2-u)^2}{2}\ \text{if $1\leqslant u\leqslant 2$}. \\
\endaligned
\right.
$$
This implies that $S_S(C)=\frac{11}{15}$ and $S(W^{C}_{\bullet, \bullet};O)=\frac{23}{30}$,
so that $\delta_P(S)\geqslant\frac{30}{23}>\frac{15}{13}$ by \eqref{equation:dP5-delta-estimate-1}.
\end{proof}

\section{Nemuro Lemma}
\label{section:nemurro}

Now, let $X$ be any smooth Fano threefold, let $\pi\colon X\to \mathbb{P}^1$ be a fibration into del Pezzo surfaces,
let~$S$ be a fiber of the~morphism $\pi$ such that $S$ is an irreducible reduced normal del Pezzo surface
that has at worst du Val singularities, and let $P$ be a point in $S$.
As in Section~\ref{section:Abban-Zhuang},
set
$$
\tau=\mathrm{sup}\Big\{u\in\mathbb{Q}_{\geqslant 0}\ \big\vert\ \text{the divisor  $-K_X-uS$ is pseudo-effective}\Big\}.
$$
For~$u\in[0,\tau]$, let $P(u)$ be the~positive part of the~Zariski decomposition of the~divisor $-K_X-uS$,
and let $N(u)$ be its negative part.
Suppose, in addition, that
$$
N(u)=\sum_{j=1}^l f_j(u) E_j
$$
for some irreducible reduced surfaces $E_1,\dots,E_l$ on the~Fano threefold $X$ that are different from~$S$,
where each $f_i\colon[0,\tau]\to\mathbb{R}_{\geqslant 0}$ is some function.
For every $j\in\{1,\ldots,l\}$, we set $c_j=\mathrm{lct}_{P}(S;E_j|_S)$.
As in  Appendix~\ref{section:delta-dP}, we set $\delta_P(S)=\delta_P(S,-K_S)$.
Define $S(W^S_{\bullet,\bullet};F)$ and $\delta_{P}(S;W^S_{\bullet,\bullet})$ as in \cite[\S~1]{ACCFKMGSSV},
or define these numbers using the~formulas used in \eqref{equation:Hamid-Ziquan}.

\begin{lemma}
\label{lemma:Nemuro}
Let $F$ be any prime divisor over $S$ such that $P\in C_S(F)$. Then
\begin{multline}
\label{equation:Nemuro}\tag{$\diamondsuit$}
\quad S\big(W^S_{\bullet,\bullet};F\big)\leqslant
A_S(F)\frac{3}{(-K_X)^3}\int_0^\tau\sum_{j=1}^\tau\frac{f_j(u)}{c_j}\big(P(u)\big|_S\big)^2du+\\
\quad +\frac{3}{(-K_X)^3}\int_0^\tau\int_0^\infty\mathrm{vol}\big(P(u)\big|_S-vF\big)dvdu\leqslant\\
\leqslant A_S(F)\Biggl(\frac{3}{(-K_X)^3}\sum_{j=1}^l \int_0^\tau \frac{f_j(u)}{c_j}\big(P(u)\big|_S\big)^2 du+\frac{3}{(-K_X)^3}\frac{\tau(-K_S)^2}{\delta_P(S)}\Biggr).\quad\quad\quad
\end{multline}
In particular, we have
\begin{eqnarray*}
\delta_{P}\big(S;W^S_{\bullet,\bullet}\big)\geqslant\Biggl(\frac{3}{(-K_X)^3}\sum_{j=1}^l \int_0^\tau \frac{f_j(u)}{c_j}\big(P(u)\big|_S\big)^2 du+\frac{3}{(-K_X)^3}\frac{\tau(-K_S)^2}{\delta_P(S)}\Biggr)^{-1}.
\end{eqnarray*}
\end{lemma}

\begin{proof}
Since the~log pair $(S, c_j E_j|_S)$ is log canonical at $P$, we conclude that $\mathrm{ord}_F(E_j|_S)\leqslant \frac{A_S(F)}{c_j}$.
Thus, we get the~first inequality in \eqref{equation:Nemuro}. Moreover, since $P(u)|_S=-K_S-N(u)|_S$, we have
$$
\int_0^\tau\int_0^\infty\mathrm{vol}(P(u)|_S-vF\big)dvdu\leqslant\int_0^\tau (-K_S)^2 S_S(F)du =\tau (-K_S)^2 S_S(F)\leqslant A_S(F) \frac{\tau (-K_S)^2}{\delta_P(S)}.
$$
Hence, the~assertion follows.
\end{proof}

\begin{corollary}
\label{corollary:Nemuro-1}
Suppose that $N(u)=0$ for every $u\in[0,\tau]$, i.e. we have $l=0$. Then
$$
\delta_P(S,W^S_{\bullet,\bullet})\geqslant\frac{(-K_X)^3\delta_P(S)}{3\tau(-K_S)^2}.
$$
\end{corollary}

\begin{corollary}
\label{corollary:Nemuro-2}
Suppose that $l=1$, $E_1|_S$ is a smooth curve contained in $S\setminus\mathrm{Sing}(S)$, and
$$
f_1(u)=
\left\{\aligned
&0\ \text{if $u\in[0,t]$}, \\
&c(u-t)\ \text{if $u\in[t,\tau]$}, \\
\endaligned
\right.
$$
for some $t\in(0,\tau)$ and some $c\in\mathbb{R}_{>0}$. Then
$$
\delta_{P}\big(S;W^S_{\bullet,\bullet}\big)\geqslant
\Biggl(
\frac{3}{(-K_X)^3}\int_t^\tau c(u-t)\big(P(u)\big\vert_S\big)^2du
+\frac{3}{(-K_X)^3}\frac{\tau(-K_S)^2}{\delta_P(S)}\Biggr)^{-1}.
$$
\end{corollary}

\end{document}